\documentclass[preprint,12pt]{elsarticle}

\usepackage[a4paper,margin=3cm]{geometry}
\usepackage{amsmath,amssymb,amsthm,mathtools}
\usepackage{hyperref}
\usepackage{enumitem}
\usepackage{graphicx}
\usepackage{tikz} 
\usepackage[nameinlink,noabbrev]{cleveref}
\usepackage{ dsfont }
\usepackage{subcaption}
\usepackage{tikz}
\usetikzlibrary{arrows.meta}
\usepackage{float}
\usetikzlibrary{decorations.pathreplacing, calc}
\usepackage[most]{tcolorbox}

\theoremstyle{plain}
\newtheorem{theorem}{Theorem}[section]
\newtheorem{proposition}[theorem]{Proposition}
\newtheorem{lemma}[theorem]{Lemma}
\newtheorem{corollary}[theorem]{Corollary}
\newtheorem{conjecture}[theorem]{Conjecture}

\theoremstyle{definition}
\newtheorem{definition}[theorem]{Definition}

\theoremstyle{remark}

\begin{document}

\begin{frontmatter}

\title{Generating Functions and the Minimum Spectral Radius in Strongly Connected Digraphs with $m+2$ Edges}

\author{Rostislav Klech}
\ead{Rostislav.Klech@math.slu.cz}
\affiliation{organization={Mathematical Institute in Opava, Silesian University in Opava},
            addressline={Na Rybníčku 626},
            city={Opava},
            postcode={746 01},
            country={Czech Republic}}

\begin{abstract}
We study the minimum adjacency spectral radius in the class
$\mathcal{SC}_{m+2}(m)$ of strongly connected digraphs with $m$ vertices and
$m+2$ edges. Using generating functions for directed paths, we associate with
the relevant digraphs topological polynomials whose smallest positive roots
determine the corresponding spectral radii. Based on an ear decomposition, we
obtain a complete structural classification of $\mathcal{SC}_{m+2}(m)$ by
showing that every digraph in this class can be obtained from a butterfly
digraph by attaching a single ear. This reduces the extremal problem to the
optimization and comparison of finitely many polynomial families subject to
their realizability conditions. We prove that the minimum spectral radius is
determined by the polynomial $P_{\min}(z)=1-2z^{m-1}-z^m$. If
$R_m\in(0,1)$ denotes its unique root, then
$\min_{G\in\mathcal{SC}_{m+2}(m)}\rho(G)=R_m^{-1}$. For $m\geq4$, the minimum
is attained, up to isomorphism, uniquely by the cross-chorded cycle
$\mathcal{C}_m^\times$. For $m=3$, there are exactly two non-isomorphic
minimizers, both with spectral radius $(1+\sqrt5)/2$. Finally, we establish the
bounds $2^{1/(m-1)}<\rho\left(\mathcal{C}_m^\times\right)<3^{1/(m-1)}$.
\end{abstract}


\end{frontmatter}

\section{Introduction}

Let $G$ be a finite strongly connected digraph with adjacency matrix $A(G)$, and let
$\rho(G)$ denote the spectral radius of $A(G)$. The spectral radius has been studied
extensively in extremal problems for strongly connected digraphs. In particular,
Lin and Shu \cite{LinShu} characterized the digraphs attaining the minimum and
maximum spectral radius among strongly connected bicyclic digraphs. Li and Zhou
\cite{LiZhou} subsequently determined the digraphs with the second, third, and
fourth smallest spectral radii among all strongly connected digraphs of fixed
order. Further extremal results and bounds under additional structural
restrictions were obtained by Hong and You \cite{HongYou}. More recently,
Shan, Wang, and He \cite{ShanWangHe} studied the more general
$\alpha$-spectral radius for several classes of strongly connected digraphs.

For a finite strongly connected digraph, the spectral radius is closely related
to the exponential growth of directed paths. If $R$ denotes the radius of
convergence of the generating function counting directed paths in $G$, then
\[
\rho(G)=R^{-1}.
\]
Equivalently, if $h(G)$ denotes the topological entropy of $G$, then
\[
h(G)=\ln \rho(G)=-\ln R.
\]
In our previous work \cite{Klech}, we used this connection to develop a
generating-function approach to the entropy of finite strongly connected
digraphs. For the class $\mathcal{SC}_{m+1}(m)$ of strongly connected digraphs
with $m$ vertices and $m+1$ edges, a butterfly parametrization reduced the
problem to the analysis of the smallest positive roots of certain topological
polynomials.

The aim of the present paper is to extend this generating-function method to
the substantially more complicated class
\[
\mathcal{SC}_{m+2}(m),
\]
consisting of strongly connected digraphs with $m$ vertices and $m+2$ edges,
and to determine the minimum adjacency spectral radius in this class. Although
only one additional edge is present compared with $\mathcal{SC}_{m+1}(m)$,
the resulting structure is considerably richer.

A central feature of our approach is that we do not begin with a prescribed
candidate for the minimizer. Instead, we first obtain a complete structural
description of the entire class $\mathcal{SC}_{m+2}(m)$. Using an ear
decomposition, we show that every digraph in this class can be obtained from
a butterfly digraph by attaching a single ear. According to the positions of
the initial and terminal vertices of the ear, this leads to a finite
classification into attachment types. For each type, we derive a corresponding
topological polynomial and optimize its parameters subject to the relevant
realizability conditions.

This complete classification is essential for the extremal result. It ensures
that all possible digraphs in $\mathcal{SC}_{m+2}(m)$ are represented in the
comparison and therefore allows us not only to identify a minimizer, but also
to determine all equality cases. In particular, it yields the uniqueness, up
to isomorphism, of the minimizer for $m\geq4$, while for $m=3$ it reveals a
second non-isomorphic digraph having the same minimum spectral radius.

The resulting extremal topological polynomial is
\[
P_{\min}(z)=1-2z^{m-1}-z^m.
\]
If $R_m\in(0,1)$ denotes its unique root in $(0,1)$, then
\[
\min_{G\in\mathcal{SC}_{m+2}(m)}\rho(G)
=
R_m^{-1}.
\]
For $m\geq4$, this minimum is attained, up to isomorphism, uniquely by the
cross-chorded cycle $\mathcal{C}_m^\times$. For $m=3$, there are exactly two
non-isomorphic minimizers,
\[
\mathcal{C}_3^\times \quad \text{and} \quad (\mathcal{W}^1_2\mathcal{W}^1_3)^-_0
\text{-}\mathcal{B}^1_{3,1},
\]
both having spectral radius
\[
\frac{1+\sqrt5}{2}.
\]
We also prove the estimates
\[
2^{\frac{1}{m-1}}
<
\rho\left(\mathcal{C}_m^\times\right)
<
3^{\frac{1}{m-1}}.
\]

The extremal digraph $\mathcal{C}_m^\times$ has previously appeared in the
spectral literature in a different notation, in particular in the work of
Shan, Wang, and He \cite{ShanWangHe}. The purpose of the present paper,
however, is different. We determine its precise extremal role within the
entire class $\mathcal{SC}_{m+2}(m)$ by means of a complete structural
classification and a generating-function analysis of all admissible types.

Besides yielding the minimum spectral radius, the complete classification also
suggests a natural candidate for the maximum spectral radius. This leads to a
conjecture, formulated at the end of the paper, both for the general setting
and for the loopless subclass.

\medskip
\noindent\textbf{Organization of the paper.} 
Section~2 introduces the necessary notation and the generating-function tools used throughout the paper. 
In Section~3, we derive the structural classification of digraphs in $\mathcal{SC}_{m+2}(m)$ by means of butterfly digraphs and ear attachments. 
Section~4 computes and optimizes the corresponding topological polynomials. 
Finally, Section~5 compares the resulting candidates and determines the minimum spectral radius together with
all equality cases.

\section{Preliminaries and notation}
For standard terminology and basic results on directed graphs, we refer to Bang-Jensen and Gutin~\cite{BangJensenGutin}. We work with finite directed graphs $G=(V,E)$ with $|V|=m$.
Loops are allowed, while multiple edges are not.
A \emph{directed path} or \emph{path} of length $n$ from $u$ to $v$ is a sequence of vertices
\[
P_{uv}=v_0\dots v_n
\]
with $v_0=u$ and $v_n=v$ such that $(v_{i-1},v_i)\in E$ for all $i=1,\dots,n$.
The \emph{length} of $P_{uv}$ is denoted by $|P_{uv}|:=n$.

The digraph $G$ is \emph{strongly connected} if for any $u,v\in V$ there exists a directed path from $u$ to $v$.
We define the \emph{out-neighborhood} of $v\in V$ by
\[
N^+(v):=\{u\in V:(v,u)\in E\},
\]
and the \emph{out-degree} of $v$ by
\[
\deg^+(v):=|N^+(v)|.
\]

If $G=(V(G),E(G))$ and $H=(V(H),E(H))$ are two digraphs, we say that $G$ and $H$ are \emph{isomorphic}, written $G\cong H$, if there exists a bijection
\[
\phi:V(G)\to V(H)
\]
such that for all $u,v\in V(G)$,
\[
(u,v)\in E(G)
\quad\Longleftrightarrow\quad
\bigl(\phi(u),\phi(v)\bigr)\in E(H).
\]

For positive integers $m$ and $q\geq m$, we denote by $\mathcal{SC}_q(m)$ the class of all strongly connected digraphs $G$ satisfying
\[
|V(G)|=m
\qquad\text{and}\qquad
|E(G)|=q.
\]
In particular, the class considered in this paper is $\mathcal{SC}_{m+2}(m)$. 

An ear $\mathcal{E}$ of length $l$ is an alternating sequence
\[
\mathcal{E}=(a_0,e_1,a_1,e_2,\ldots,e_l,a_l),
\]
where $e_1,\ldots,e_l$ are vertices and $a_0,\ldots,a_l$ are directed edges
such that the terminal vertex of $a_{i-1}$ and the initial vertex of $a_i$
are both $e_i$ for every $i=1,\ldots,l$. The vertices $e_1,\ldots,e_l$ are called the internal vertices of $\mathcal{E}$.
We define the length of $\mathcal{E}$ to be the number $l$ of its internal vertices.
Thus, an ear of length $l$ contains $l$ internal vertices and $l+1$ edges.
When an ear $\mathcal{E}$ is inserted from a vertex $u$ to a vertex $v$, the initial
vertex of its first edge is identified with $u$, while the terminal vertex
of its last edge is identified with $v$.

We denote by $\mathcal{C}_m$ the directed cycle with
\begin{align*}
V(\mathcal{C}_m)
&=\{v_1,v_2,\dots,v_m\},\\
E(\mathcal{C}_m)
&=\bigl\{(v_i,v_{i+1}):i\in\{1,2,\dots,m-1\}\bigr\}
  \cup\{(v_m,v_1)\}.
\end{align*}

Let $p(n)$ denote the total number of paths of length $n$ in $G$, i.e.,
\[
p(n)
:=
\bigl|
\{(v_0,\dots,v_n)\in V^{n+1}:
(v_{i-1},v_i)\in E
\text{ for all }i=1,\dots,n\}
\bigr|.
\]
We define the corresponding generating function by
\[
P(z):=\sum_{n\geq 0}p(n)z^n.
\]
The topological entropy of $G$ is defined by
\[
h(G):=
\limsup_{n\to\infty}
\frac{1}{n}\ln\bigl(p(n)\bigr).
\]
For background on generating functions and their analytic properties, we refer to Flajolet and Sedgewick~\cite{FlajoletSedgewick} and Wilf~\cite{Wilf}.

\begin{theorem}[Pringsheim’s theorem]\label{pring}
Let
\[
f(z)=\sum_{n=0}^{\infty} a_n z^n
\]
be a power series with $a_n\ge 0$ for all $n$ and radius of convergence $R\in(0,\infty)$.
Then $z=R$ is a singular point of $f$.
\end{theorem}

\begin{proposition}\label{spec_rad}
Let $G$ be a strongly connected digraph, let $A(G)$ denote its adjacency matrix, and let $\rho(G)$ be the spectral radius of $A(G)$. Let $R$ denote the radius of convergence of the generating function
\[
P(z)=\sum_{n=0}^{\infty} p(n)z^n.
\]
Then
\[
\rho(G)=R^{-1}.
\]
\end{proposition}

\begin{proof}
It is well known that the topological entropy $h(G)$ of a strongly connected digraph $G$ admits the equivalent representations
\[
h(G)=\ln \rho(G)=-\ln R.
\]
Therefore,
\[
\rho(G)=R^{-1}.
\]
\end{proof}

\begin{lemma}[\cite{Klech}, Lemma 3.3.]
    Let $G$ be a strongly connected digraph with $m$ vertices. For each vertex $v_i$, let $P_i(z)$ be the generating function for the number of paths in $G$ that start at $v_i$. Then all $P_i(z)$ have the same radius of convergence.
\end{lemma}

\begin{corollary}[\cite{Klech}, Corollary 3.4.]
    Let $G$ be a strongly connected digraph with $m$ vertices. Then the generating functions
    $P_1(z),\dots,P_m(z)$ all have the same radius of convergence, say $R$.
    Moreover, the total generating function
    \[
    P(z)=\sum_{i=1}^m P_i(z)
    \]
    has radius of convergence $R$ as well. In particular, to determine $R$ (and hence $\rho(G)$)
    it suffices to compute $P_i(z)$ for any fixed vertex $v_i$.
\end{corollary}

\begin{lemma}[\cite{Klech}, Lemma 3.6.]\label{cycle_rad}
    Let $G$ be a strongly connected digraph on $m$ vertices and let
\[
P(z)=\sum_{n=0}^{\infty} p(n)\,z^n
\]
be the generating function for the total number of paths of length $n$ in $G$.
Let $R$ be the radius of convergence of $P(z)$. Then
\[
R\in(0,1].
\]
Moreover,
\[
R=1 \Longleftrightarrow G\cong \mathcal{C}_m.
\]
\end{lemma}

\begin{definition}\label{def:topological_polynomial}
Let $G$ be a strongly connected digraph. Let
\[
P(z)=\sum_{n=0}^{\infty} p(n)\,z^n
\]
be the generating function, where $p(n)$ denotes the number of paths of length $n$, and let $R\in(0,1]$ be the radius of convergence of $P(z)$.

A polynomial $T(G;z)\in\mathbb{Z}[z]$ is called a \emph{topological polynomial of $G$}
if it satisfies the following two conditions:
\begin{enumerate}
\item $T(G;R)=0$,
\item $T(G;z)\neq 0$ for every $z\in(0,R)$.
\end{enumerate}

The equation
\[
T(G;z)=0
\]
is called the \emph{topological equation} of $G$.
\end{definition}

\begin{lemma}[\cite{Klech}, Lemma 3.8.]\label{recurence}
Let $G$ be a strongly connected digraph and suppose that there exists a sequence of vertices
\[
(v_1,v_2,\dots,v_i)
\]
such that 
\[
N^+(v_j)=\{v_{j+1}\}\quad \forall j\in\{1,2,\dots,i-1\}.
\]
Then
\[
P_1(z)=\frac{1-z^{\,i-1}}{1-z}+z^{\,i-1}P_i(z).
\]
\end{lemma}

\begin{lemma}\label{lem:topological_equation_from_vertex_gf}
Let $G=(V,E)\in\mathcal{SC}_{m+2}(m)$, and let $u\in V$.
Let
\[
U(z)=\sum_{n=0}^{\infty}p_u(n)z^n
\]
be the generating function for the number $p_u(n)$ of directed paths of
length $n$ starting at $u$. Suppose that $U(z)$ can be written in the form
\[
U(z)=\frac{Q(z)}{(1-z)\Phi(z)},
\]
where $Q(z)$ is a polynomial and $\Phi(z)$ is a polynomial satisfying
\[
\Phi(0)>0,\qquad \Phi(1)<0,
\]
and assume that $\Phi$ is strictly decreasing on $(0,1)$. Then the equation
\[
\Phi(z)=0
\]
is the topological equation of $G$.
\end{lemma}

\begin{proof}
The coefficients of $U(z)$ are nonnegative. Hence, by Pringsheim's theorem,
the point $z=R$ is a singular point of $U(z)$. On the other hand, by
assumption,
\[
U(z)=\frac{Q(z)}{(1-z)\Phi(z)}.
\]
Since $R\in(0,1)$, the factor $1-z$ does not vanish at $z=R$.
Therefore, the singularity of $U(z)$ at $z=R$ must come from the factor
$\Phi(z)$. Hence
\[
\Phi(R)=0.
\]

Now, since $\Phi(0)>0$, $\Phi(1)<0$, and $\Phi$ is strictly decreasing on
$(0,1)$, the equation
\[
\Phi(z)=0
\]
has a unique solution in the interval $(0,1)$. We have just shown that this
solution is precisely $z=R$, the radius of convergence of $U(z)$.

Thus, the equation
\[
\Phi(z)=0
\]
determines the radius of convergence associated with $G$, and hence it is the
topological equation of $G$.
\end{proof}

\section{Structural Characterization and Decomposition of the class $\mathcal{SC}_{m+2}(m)$}

\begin{definition}
Let $G\in\mathcal{SC}_{m+1}(m)$ and $k_1,k_2 \in \{1,2,\dots\}$.  We say that
$G$ is a \emph{$(t,k_1,k_2)$-butterfly digraph} $\mathcal{B}^{\,t}_{k_1,k_2}$ if
\[
  k_1 + k_2 - t = m,\qquad
  1 \le t \le \min\{k_1,k_2\},
\]
and there exist two distinct directed simple cycles $\mathcal{C}_{k_1},\mathcal{C}_{k_2}$ as subgraphs of $G$ satisfying
\[
  \left|V(\mathcal{C}_{k_1})\right| = k_1,\quad
  \left|V(\mathcal{C}_{k_2})\right| = k_2,\quad
  \bigl|V(\mathcal{C}_{k_1})\cap V(\mathcal{C}_{k_2})\bigr| = t.
\]
\end{definition}

\begin{lemma}[\cite{Klech}, Lemma 4.3.]\label{lem:m+1_is_butterfly}
Let $G\in\mathcal{SC}_{m+1}(m)$.
Then there exist integers $k_1,k_2\in\mathds{N}$ and $t\in\mathds{N}$ such that
\[
G\cong \mathcal{B}^{\,t}_{k_1,k_2}.
\]
\end{lemma}

\begin{figure}[H]
    \centering
    \includegraphics[width=0.8\textwidth]{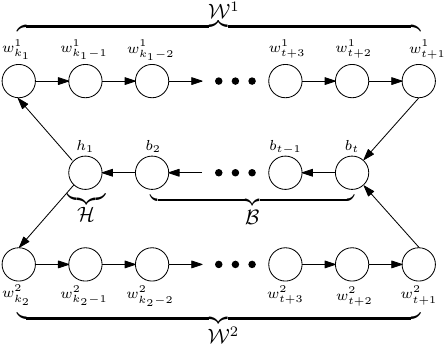} 
    \caption{Head ($\mathcal{H}$), Wings ($\mathcal{W}^1, \mathcal{W}^2$) and Body ($\mathcal{B}$) of $\mathcal{B}^{\,t}_{k_1,k_2}$.}
    \label{fig:arb_butterfly}
\end{figure}

\begin{theorem}[\cite{BangJensenGutin}]\label{ucha}
    Let $G$ be a digraph with at least two vertices. Then $G$ is strongly connected if and only if it admits an ear decomposition.
\end{theorem}

\begin{corollary}\label{butterfly-core}
Let $G\in\mathcal{SC}_{m+2}(m)$. Then $G$ can be obtained from a
butterfly digraph by attaching an ear.
\end{corollary}

\begin{proof}
By Theorem~\ref{ucha}, $G$ admits an ear decomposition.
Let $\mathcal{E}$ be the last ear in this decomposition, and suppose that
$\mathcal{E}$ has $l$ internal vertices. Removing the edges and the
internal vertices of $\mathcal{E}$, we obtain a strongly connected digraph
$G'$ satisfying
\[
|V(G')|=m-l
\qquad\text{and}\qquad
|E(G')|=(m+2)-(l+1)=m-l+1.
\]
Hence, writing $m'=m-l$, we have
\[
G'\in\mathcal{SC}_{m'+1}(m').
\]
By \cite[Lemma~4.3.]{Klech}, it follows that
\[
G'\cong\mathcal{B}_{k_1,k_2}^{\,t}
\]
for some admissible parameters $t,k_1,k_2$.
\end{proof}

To analyze the spectral radius of digraphs in $\mathcal{SC}_{m+2}(m)$, we classify the possible positions of the ear $\mathcal{E}$ relative to the butterfly core $\mathcal{B}_{k_1, k_2}^{\,t}$.
We partition the vertex set $V(\mathcal{B}_{k_1, k_2}^{\,t})$ into four regions (subsets) as in Figure~\ref{fig:arb_butterfly}:
\begin{enumerate}
    \item The \textit{First Wing} ($\mathcal{W}^1$): vertices exclusive to the first cycle ($|\mathcal{W}^1| = k_1 - t$).
    \item The \textit{Second Wing} ($\mathcal{W}^2$): vertices exclusive to the second cycle ($|\mathcal{W}^2| = k_2 - t$).
    \item The \textit{Head} ($\mathcal{H}$): the only vertex with out-degree $2$ ($|\mathcal{H}|=1$).
    \item The \textit{Body} ($\mathcal{B}$): vertices in the intersection with out-degree $1$ ($|\mathcal{B}| = t-1$).
\end{enumerate}
We now list all possible ordered types of pairs of regions:
\[
\begin{aligned}
&
\mathcal{W}^1\mathcal{W}^1,\ 
\mathcal{W}^1\mathcal{W}^2,\ 
\mathcal{W}^1\mathcal{B},\ 
\mathcal{W}^1\mathcal{H},\\
&
\mathcal{W}^2\mathcal{W}^1,\ 
\mathcal{W}^2\mathcal{W}^2,\ 
\mathcal{W}^2\mathcal{B},\ 
\mathcal{W}^2\mathcal{H},\\
&
\mathcal{B}\mathcal{W}^1,\ 
\mathcal{B}\mathcal{W}^2,\ 
\mathcal{B}\mathcal{B},\ 
\mathcal{B}\mathcal{H},\\
&
\mathcal{H}\mathcal{W}^1,\ 
\mathcal{H}\mathcal{W}^2,\ 
\mathcal{H}\mathcal{B},\ 
\mathcal{H}\mathcal{H}.
\end{aligned}
\]

Let 
\[
\mathcal{\mathcal{X}},\mathcal{Y}\in\{\mathcal{W}^1,\mathcal{W}^2,\mathcal{B},\mathcal{H}\}.
\]
If $x_i\in \mathcal{X}$ and $y_j\in \mathcal{Y}$, then by
\[
\left(\mathcal{X}_i\mathcal{Y}_j\right)_l\text{-}\mathcal{B}_{k_1,k_2}^{\,t}
\]
we denote the digraph obtained from the butterfly digraph
$\mathcal{B}_{k_1,k_2}^{\,t}$ by inserting an ear $\mathcal{E}$ of length $l$ from the vertex $\ x_i$ to the vertex $y_j$.

\begin{figure}[H]
    \centering
    \includegraphics[width=0.6\textwidth]{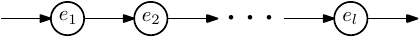} 
    \caption{Ear $\mathcal{E}$ of length $l$.}
    \label{fig:ear}
\end{figure}

For clarity, we also distinguish the above types according to whether the ear is inserted between two different regions or within the same region. Types of the first kind will be called \emph{cross-region types}, while types of the second kind will be called \emph{intra-region types}. For intra-region types, we must further distinguish whether the ear is inserted
in the direction of the edges of the original region or in the opposite direction.
These two situations lead to different recurrence relations, from which we will
derive the corresponding generating functions. 

Thus, let $\mathcal{X}$ be a region and let $x_i,x_j\in \mathcal{X}$. The digraph obtained from
$\mathcal{B}_{k_1,k_2}^{\,t}$ by inserting an ear $\mathcal{E}$ of length
$l$ from the vertex $x_i$ to the vertex $x_j$ in the direction of the edges
of the original region will be denoted by
\[
\left(\mathcal{X}_i\mathcal{X}_j\right)_l^+\text{-}\mathcal{B}_{k_1,k_2}^{\,t}.
\]
Conversely, the digraph obtained by inserting an ear $\mathcal{E}$ of length
$l$ from the vertex $x_i$ to the vertex $x_j$ against the direction of the
edges of the original region will be denoted by
\[
\left(\mathcal{X}_i\mathcal{X}_j\right)_l^-\text{-}\mathcal{B}_{k_1,k_2}^{\,t}.
\]
Both situations are illustrated in Figures~\ref{fig:+dir} and \ref{fig:-dir}, where the wavy edge denotes the ear introduced in Figure~\ref{fig:ear}.
The ear can also be inserted as a cycle, that is, from a vertex $x_i$ back to the same vertex $x_i$. We denote this case by
\[
\left(\mathcal{X}_i\mathcal{X}_i\right)_l^0
\text{-}\mathcal{B}_{k_1,k_2}^{\,t}.
\]

\begin{figure}[H]
    \centering
    \includegraphics[width=0.6\textwidth]{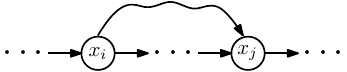} 
    \caption{Insertion of an ear in the direction of the original region $\mathcal{X}$.}
    \label{fig:+dir}
\end{figure}

\begin{figure}[H]
    \centering
    \includegraphics[width=0.6\textwidth]{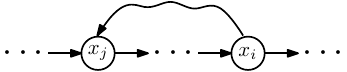} 
    \caption{Insertion of an ear against the direction of the original region $\mathcal{X}$.}
    \label{fig:-dir}
\end{figure}

Observe that for the type
\[
\left(\mathcal{H}_i\mathcal{H}_j\right)_l\text{-}\mathcal{B}_{k_1,k_2}^{\,t}
\]
we necessarily have $i=j=1$, since $\mathcal{H}=\{h_1\}$. Hence, in this case,
there is no need to distinguish the direction in which the ear is inserted. For the same reason, whenever a type involves the region $\mathcal{H}$, the
corresponding index $i$, respectively $j$, is redundant. Hence, no subscript
will be attached to $\mathcal{H}$ in our notation.
\\
\\
\textbf{Remark.}
For specific digraphs of type
\[
\left(\mathcal{X}_i\mathcal{X}_j\right)_l\text{-}\mathcal{B}_{k_1,k_2}^{\,t},
\]
that is, for fixed values of $i$ and $j$, the direction in which the ear is inserted can also be determined from the relation between the parameters $i$ and $j$. If $i>j$, then the ear is inserted in the direction of the original region. Conversely, if $i<j$, then the ear is inserted against the direction of the original region. If $i=j$, then the inserted ear forms a new cycle from the vertex $\mathcal{X}_i$ back to itself. Since, however, we are currently working with general parameters, we use the superscripts $+$, $-$ and $0$ to indicate the direction explicitly.
\\

In the following summary, we omit the suffix
$\text{-}\mathcal{B}_{k_1,k_2}^{\,t}$ from the notation, since it is the same
for all types considered here. We now list all possible types of digraphs
obtained by inserting an ear of length $l$ into
$\mathcal{B}_{k_1,k_2}^{\,t}$ in the following table.

\begin{table}[H]
\centering
\renewcommand{\arraystretch}{1.2}
\[
\begin{array}{c|c}
\text{Cross-region types} & \text{Intra-region types} \\
\hline
\hline
(\mathcal{W}^1_i\mathcal{W}^2_j)_l & (\mathcal{W}^1_i\mathcal{W}^1_j)^+_l \\
(\mathcal{W}^2_i\mathcal{W}^1_j)_l & (\mathcal{W}^1_i\mathcal{W}^1_j)^-_l \\
(\mathcal{W}^1_i\mathcal{B}_j)_l & (\mathcal{W}^2_i\mathcal{W}^2_j)^+_l \\
(\mathcal{W}^2_i\mathcal{B}_j)_l & (\mathcal{W}^2_i\mathcal{W}^2_j)^-_l \\
(\mathcal{W}^1_i\mathcal{H})_l & (\mathcal{B}_i\mathcal{B}_j)^+_l \\
(\mathcal{W}^2_i\mathcal{H})_l & (\mathcal{B}_i\mathcal{B}_j)^-_l \\
(\mathcal{B}_i\mathcal{W}^1_j)_l & (\mathcal{H}\mathcal{H})_l \\
(\mathcal{B}_i\mathcal{W}^2_j)_l &  (\mathcal{W}^1_i\mathcal{W}^1_i)^0_l\\
(\mathcal{B}_i\mathcal{H})_l &  (\mathcal{W}^2_i\mathcal{W}^2_i)^0_l\\
(\mathcal{H}\mathcal{W}^1_j)_l &  (\mathcal{B}_i\mathcal{B}_i)^0_l\\
(\mathcal{H}\mathcal{W}^2_j)_l &  \\
(\mathcal{H}\mathcal{B}_j)_l & 
\end{array}
\]
\caption{A combinatorial classification of ear insertions of length $l$.}
\label{tab:ear-types}
\end{table}

\begin{lemma}
    Let $G\in\mathcal{SC}_{m+2}(m)$, and let $l\geq 0$. By the classification
    above, $G$ is obtained from a butterfly core
    $\mathcal{B}_{k_1,k_2}^{\,t}$ by inserting an ear of length $l$. Then the
    parameters $t,k_1,k_2$, and $l$ satisfy
    \begin{equation}\label{params}
        k_1+k_2-t+l=m.
    \end{equation}
\end{lemma}

\begin{proof}
The core of $G$ is the butterfly digraph
$\mathcal{B}_{k_1,k_2}^{\,t}$. Its number of vertices is
\[
    k_1+k_2-t.
\]
Since $G$ is obtained from this core by inserting an ear of length $l$,
and this insertion adds exactly $l$ new vertices, the resulting digraph has
\[
    k_1+k_2-t+l
\]
vertices. On the other hand, $G\in\mathcal{SC}_{m+2}(m)$, and hence $G$
has exactly $m$ vertices. Therefore
\[
    k_1+k_2-t+l=m.
\]

\end{proof}

\section{Topological polynomials associated with digraph types in $\mathcal{SC}_{m+2}(m)$}
In what follows, we again omit the suffix $\text{-}\mathcal{B}_{k_1,k_2}^{\,t}$ ($\text{-}\mathcal{B}_{k_s,k_r}^{\,t}$) from the notation. We also refer the reader to Figure~\ref{fig:arb_butterfly}, which provides the reference structure for the butterfly core. The figures corresponding to the individual types are only schematic, and this reference figure should help the reader identify the relevant parts of the core. In the following figures, the wavy edge denotes the ear introduced in Figure~\ref{fig:ear}. 

The topological equation is derived from the recurrence relations associated
with the individual vertices of the digraph. For a vertex $u_i^r$, we denote
the corresponding generating function, and hence its recurrence relation,
by the capital letter with the same superscript and subscript, that is, by
$U_i^r$.

As described in detail in our previous work~\cite[Section~3]{Klech},
these recurrence relations are obtained directly from the outgoing edges of
each vertex. In general, if a vertex $u$ has out-neighbors
$v_1,\ldots,v_d$, and if $U,V_1,\ldots,V_d$ denote the corresponding
generating functions, then
\[
U=1+zV_1+\cdots+zV_d.
\]
Thus, each outgoing edge from $u$ to a vertex $v_j$ contributes the term
$zV_j$ to the recurrence relation associated with $u$. We use this rule
throughout this section without further comment.

For each type, it is also necessary to determine its realizability conditions.
Thus, let
\[
    \left(\mathcal{X}_i\mathcal{Y}_j\right)_l
\]
be one of the types described above. In order for this type to be realizable, we require
\[
    |\mathcal{X}|\geq 1
    \qquad\text{and}\qquad
    |\mathcal{Y}|\geq 1.
\]
In some cases, stronger conditions may be needed. Whenever this occurs, we will state these additional requirements explicitly.

We now describe the general procedure that will be used for each type:
\begin{enumerate}
    \item We provide a schematic figure of the given type.

    \item Using Lemma~\ref{recurence}, we derive the recurrence relations associated
    with the given type.

    \item We solve these recurrence relations with respect to the vertex $h_1$.
    In this way, we obtain the generating function $H_1(z)$. We then use
    Lemma~\ref{lem:topological_equation_from_vertex_gf} to identify the
    corresponding topological polynomial. The verification that $H_1(z)$ has
    the required form of the function $U(z)$ from
    Lemma~\ref{lem:topological_equation_from_vertex_gf} is routine and will not
    be repeated for each type separately. Indeed, in all cases considered below,
    the required form of $H_1(z)$, as well as the corresponding properties of
    the polynomial $\Phi(z)$, are apparent from the obtained expression.

    \item We determine the topological polynomial $\Phi(z)$.

    \item We perform the first $(i,j)$-optimization step for the parameters. More precisely,
    we optimize the parameters $i$ and $j$, which determine the positions of
    the inserted ear. After performing this optimization, we will indicate the relevant realizability conditions alongside each corresponding $(i,j)$\textit{-optimized} topological polynomial.

    \item Under the assumption $k_1\geq k_2$, we perform the second $(s,r)$-optimization
    step for the parameters $k_s$ and $k_r$. In this step, we optimize the choice of
    $s$ and $r$, where $s,r\in\{1,2\}$. The result will be an $(s,r,i,j)$\textit{-optimized} topological polynomial.
\end{enumerate}

\noindent\textbf{Remark.} Although the regions $\mathcal{H}$ and $\mathcal{B}$ together form the common part of the two
cycles of the butterfly digraph, we treat them separately in the
classification. When $\mathcal{H}$ or $\mathcal{B}$ occurs as the terminal region of an ear,
the corresponding types can often be described by a common recurrence
pattern, with $h_1$ playing the role of a boundary vertex of the body.
However, this is no longer the case when the ear starts at $h_1$.
Indeed, since $h_1$ is the unique vertex of out-degree $2$, an ear starting
at $h_1$ introduces an additional term directly into the recurrence relation
for $H_1$, whereas an ear starting at a vertex $b_i\in B$ modifies the
recurrence relation associated with $B_i$. We therefore keep $\mathcal{H}$ and $\mathcal{B}$
as distinct regions throughout the analysis, which avoids introducing
additional exceptional cases into the individual attachment types.

\subsection{Cross-region types}

In this section, we gradually derive the $(s,r,i,j)$-optimized
topological polynomials for the cross-region types. We begin with the type
\[
    (\mathcal{W}^s_{i}\mathcal{B}_{j})_l
\]
which serves as a model example for the general optimization procedure. Since
the arguments used to justify the optimal choices of parameters are almost the
same for all cross-region types, we provide the full optimization argument only
in this first case. In the subsequent cases, we record the relevant conditions,
the optimal parameter choices, and the resulting optimized topological
polynomial, without repeating the same reasoning in detail.

\medskip
Consider the topological polynomial associated with digraphs of type
\begin{equation*}
    (\mathcal{W}^s_{i}\mathcal{B}_{j})_l, \quad s\in\{1,2\}.
\end{equation*}

\begin{figure}[ht]
    \centering
    \includegraphics[width=0.5\textwidth]{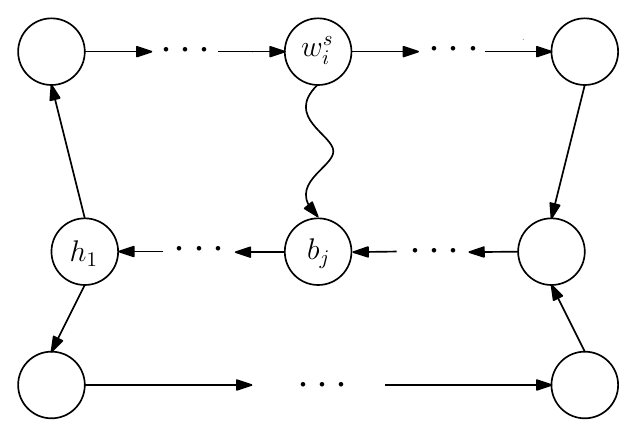}
    \caption{
    \((\mathcal{W}^s_{i}\mathcal{B}_{j})_l
    \text{-}\mathcal{B}_{k_s,k_r}^{\,t}\).}
    \label{fig:WB}
\end{figure}

\medskip

\noindent
\begin{minipage}[t]{0.4\textwidth}
\small
\textit{Recurrence relations.}
\begin{equation*}
\begin{aligned}
    H_1&=1+zW^s_{k_s}+zW^r_{k_r}, \\
    W^s_{k_s}&=\frac{1-z^{k_s-i}}{1-z}+z^{k_s-i}W^s_i, \\
    W^s_i&=1+zW^s_{i-1}+zE_1, \\
    E_1&=\frac{1-z^l}{1-z}+z^lB_j, \\
    W^s_{i-1}&=\frac{1-z^{i-2}}{1-z}+z^{i-2}H_1, \\
    W^r_{k_r}&=\frac{1-z^{k_r-1}}{1-z}+z^{k_r-1}H_1, \\
    B_j&=\frac{1-z^{j-1}}{1-z}+z^{j-1}H_1.
\end{aligned}
\end{equation*}
\end{minipage}
\hfill
\begin{minipage}[t]{0.6\textwidth}
\small
\textit{Solution for $H_1$.}
\begin{equation*}
H_1
=
\frac{
1+z-z^{k_s}-z^{k_r}
+z^{k_s-i+2}
-z^{k_s-i+l+j+1}
}{
(1-z)\left(
1-z^{k_s}-z^{k_r}
-z^{k_s-i+l+j+1}
\right)
}.
\end{equation*}

\vspace{1em}

\textit{Topological polynomial.}
\begin{equation*}
    T\left((\mathcal{W}^s_i\mathcal{B}_j)_l;z\right)
    =
    1-z^{k_s}-z^{k_r}
    -z^{k_s-i+l+j+1}.
\end{equation*}
\end{minipage}

\medskip

Since $\rho(G)=R^{-1}$, minimizing the spectral radius is equivalent to maximizing the root $R\in(0,1)$ of the corresponding topological equation. Therefore, we need to maximize the exponent
\[
    k_s-i+l+j+1.
\]
Equivalently, we need to maximize the difference $j-i$. Since
\[
i\in\{k_s,k_s-1,\dots,t+1\} \quad \text{and} \quad j\in\{t,t-1,\dots,2\},
\]
the optimal choice of the parameters is $i=t+1$ and $j=t$. However, this choice is valid only when $l\neq 0$. If $l=0$, then the inserted ear is just a single edge, and the above choice of $i$ and $j$ is not possible, since the edge joining the vertices $w^s_{t+1}$ and $b_t$ already belongs to the original digraph
$\mathcal{B}_{k_s,k_r}^{\,t}$ (we will use this observation for several other types below, without mentioning it explicitly each time). Therefore, one of the indices has to be shifted.

In this case, the optimal choices are
\[
(i,j)=(t+2,t)
\qquad\text{or}\qquad
(i,j)=(t+1,t-1).
\]
Both choices give $j-i=-2$ and hence they are equivalent with respect to their contribution to the exponent $k_s-i+l+j+1$. The realizability conditions, however, are different:
\begin{align*}
    (i,j)&=(t+2,t): && |\mathcal{W}^s|\geq 2,\; |\mathcal{B}|\geq 1\\
    (i,j)&=(t+1,t-1): && |\mathcal{W}^s|\geq1,\; |\mathcal{B}|\geq 2.
\end{align*}

Thus, we obtain the following $(i,j)$-optimized equations. For $l\neq 0$, we have
\begin{equation}\label{2}
    T\left((\mathcal{W}^s_{t+1}\mathcal{B}_t)_{l\neq0};z\right)
    =
    1-z^{k_s}-z^{k_r}-z^{k_s+l}, \quad |\mathcal{W}^s|\geq 1,\; |\mathcal{B}|\geq 1. 
\end{equation}
For $l=0$, we obtain
\begin{equation}
\begin{aligned}\label{(3)}
    T\left((\mathcal{W}^s_{t+2}\mathcal{B}_t)_0;z\right) =&\;1-z^{k_s}-z^{k_r}-z^{k_s-1}, \quad|\mathcal{W}^s|\geq 2,\; |\mathcal{B}|\geq 1,  \\
    T\left((\mathcal{W}^s_{t+1}\mathcal{B}_{t-1})_0;z\right) =&\;1-z^{k_s}-z^{k_r}-z^{k_s-1}, \quad|\mathcal{W}^s|\geq 1,\; |\mathcal{B}|\geq 2.
\end{aligned}
\end{equation}

We now optimize the parameters $s$ and $r$. In the last term of
equations \eqref{2} and \eqref{(3)}, the parameter $k_s$ appears in the
exponent with a positive contribution. Since we want this exponent to be maximized, the assumption $k_1\geq k_2$ implies that the optimal
choice is $s=1$ and $r=2$. Hence, for $l\neq 0$, we have
\begin{equation*}
    T\left((\mathcal{W}^1_{t+1}\mathcal{B}_t)_{l\neq0};z\right)
    =
    1-z^{k_1}-z^{k_2}-z^{k_1+l}=0, \quad |\mathcal{W}^1|\geq 1,\; |\mathcal{B}|\geq 1.
\end{equation*}
For $l=0$, we obtain
\begin{equation*}
\begin{aligned}
    T\left((\mathcal{W}^1_{t+2}\mathcal{B}_t)_0;z\right) =&\;1-z^{k_1}-z^{k_2}-z^{k_1-1}=0, \quad|\mathcal{W}^1|\geq 2,\; |\mathcal{B}|\geq 1,  \\
    T\left((\mathcal{W}^1_{t+1}\mathcal{B}_{t-1})_0;z\right) =&\;1-z^{k_1}-z^{k_2}-z^{k_1-1}=0, \quad|\mathcal{W}^1|\geq 1,\; |\mathcal{B}|\geq 2.
\end{aligned}
\end{equation*}

\begin{lemma}\label{lem:WB_l=0}
For $l\geq 1$, we have
\[\label{eq:izom-WB}
(\mathcal{W}^1_{t+1}\mathcal{B}_t)_{l}\text{-}\mathcal{B}_{k_1,k_2}^{\,t}
\cong
(\mathcal{W}^1_{t+1+l}\mathcal{B}_t)_{0}\text{-}\mathcal{B}_{k_1+l,k_2}^{\,t}.
\]
\end{lemma}

\begin{proof}
Let
\[
G=(\mathcal{W}^1_{t+1}\mathcal{B}_t)_{l}\text{-}\mathcal{B}_{k_1,k_2}^{\,t}
\]
and
\[
H=(\mathcal{W}^1_{t+1+l}\mathcal{B}_t)_{0}\text{-}\mathcal{B}_{k_1+l,k_2}^{\,t}.
\]
We define a map
\[
\phi:V(G)\to V(H)
\]
by
\[
\phi(v)=
\begin{cases}
w^1_{\alpha+l}, & \text{if } v=w^1_\alpha,\ \alpha=t+1,t+2,\dots,k_1,\\
w^1_{t+l+1-\alpha}, & \text{if } v=e_\alpha,\ \alpha=1,\dots,l,\\
v, & \text{otherwise}.
\end{cases}
\]
The map $\phi$ is bijective. Moreover, it maps the vertices 
\[
e_l,e_{l-1},\dots,e_1
\]
onto
\[
w^1_{t+1},w^1_{t+2},\dots,w^1_{t+l},
\]
respectively, and shifts the original vertices $w^1_{k_1},\dots,w^1_{t+1}$ to
$w^1_{k_1+l},\dots,w^1_{t+l+1}$. All remaining vertices are fixed.

It follows directly from the definition of the edges in the two digraphs that
\[
(u,v)\in E(G)
\quad\Longleftrightarrow\quad
(\phi(u),\phi(v))\in E(H).
\]
Hence $\phi$ is a digraph isomorphism, and therefore $G\cong H$.
\end{proof}

\begin{corollary}\label{cor:WB_l0}
To find the minimum spectral radius in the class of digraphs
\[
(\mathcal{W}_i^s\mathcal{B}_j)_l,
\]
it is sufficient to consider only the case $l=0$.
\end{corollary}

\begin{proof}
    By Lemma~\ref{lem:WB_l=0}, for every $l\geq 1$ we have
    \[
    (\mathcal{W}^1_{t+1}\mathcal{B}_t)_{l}\text{-}\mathcal{B}_{k_1,k_2}^{\,t}
    \cong
    (\mathcal{W}^1_{t+1+l}\mathcal{B}_t)_{0}\text{-}\mathcal{B}_{k_1+l,k_2}^{\,t}.
    \]
    Since isomorphic digraphs have the same spectral radius, we may replace a digraph with $l\geq 1$ by its isomorphic representative with ear of length $0$.

    Suppose that
    \[
    (\mathcal{W}^1_{t+1}\mathcal{B}_t)_{l}\text{-}\mathcal{B}_{k_1,k_2}^{\,t}
    \]
    is an $(s,r,i,j)$-optimized digraph for $l\neq 0$. Its isomorphic representative
    \[
    (\mathcal{W}^1_{t+1+l}\mathcal{B}_t)_{0}\text{-}\mathcal{B}_{k_1+l,k_2}^{\,t}
    \]
    is obtained by inserting an ear of length $0$ between the vertices $w^1_i$ and $b_j$, where
    \[
    i=t+1+l
    \qquad\text{and}\qquad
    j=t.
    \]

    However, we have shown above that, in the case of an ear of length $0$, the optimal choice for attaining the minimum is either
    \[
    (i,j)_1=(t+2,t) \quad \text{or} \quad (i,j)_2=(t+1,t-1).
    \]
    If $l>1$, then
    \[
    (i,j)=(t+1+l,t)
    \]
    is not one of these optimal choices. Hence the corresponding digraph with $l>1$ cannot attain a smaller spectral radius than the minimum already attained in the subclass with $l=0$.

    In the case $l=1$, we obtain
    \[
    (i,j)=(t+2,t)=(i,j)_1,
    \]
    which is precisely one of the optimal choices for $l=0$. Thus this case is also already included in the optimization for $l=0$.

    Therefore, in order to find the minimum spectral radius in the given class, it is sufficient to consider only the case $l=0$.
\end{proof}
\medskip

Next, consider the type 
\begin{equation*}
    (\mathcal{W}^s_i\mathcal{W}^r_j)_l, \quad s,r\in\{1,2\}, \quad s\neq r.
\end{equation*}

\begin{figure}[ht]
    \centering
    \includegraphics[width=0.5\textwidth]{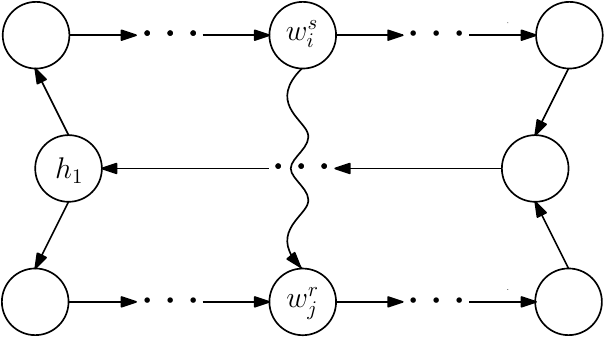}
    \caption{$(\mathcal{W}^s_i\mathcal{W}^r_j)_l\text{-}\mathcal{B}_{k_s,k_r}^{\,t}$.}
    \label{fig:WW}
\end{figure}

\medskip

\noindent
\begin{minipage}[t]{0.4\textwidth}
\small
\textit{Recurrence relations.}
\[
\begin{aligned}
    H_1&=1+zW^s_{k_s}+zW^r_{k_r}, \\
    W^s_{k_s}&=\frac{1-z^{k_s-i}}{1-z}+z^{k_s-i}W^s_i, \\
    W^s_i&=1+zW^s_{i-1}+zE_1, \\
    E_1&=\frac{1-z^l}{1-z}+z^lW^r_j, \\
    W^s_{i-1}&=\frac{1-z^{i-2}}{1-z}+z^{i-2}H_1, \\
    W^r_{k_r}&=\frac{1-z^{k_r-1}}{1-z}+z^{k_r-1}H_1, \\
    W^r_j&=\frac{1-z^{j-1}}{1-z}+z^{j-1}H_1.
\end{aligned}
\]
\end{minipage}
\hfill
\begin{minipage}[t]{0.6\textwidth}
\small
\textit{Solution for $H_1$.}
\begin{equation*}
H_1
=
\frac{
1+z-z^{k_s}-z^{k_r}
+z^{k_s-i+2}-z^{k_s-i+l+j+1}
}{
(1-z)\left(
1-z^{k_s}-z^{k_r}
-z^{k_s-i+l+j+1}
\right)
}.
\end{equation*}

\vspace{1em}

\textit{Topological polynomial.}
\begin{equation*}
    T\left((\mathcal{W}^s_i\mathcal{W}^r_j)_l;z\right) =
1-z^{k_s}-z^{k_r}
-z^{k_s-i+l+j+1}.
\end{equation*}
\end{minipage}

\medskip

\noindent
\begin{minipage}[t]{0.48\textwidth}
\small
\noindent
$(i,j)$\textit{-optimization.}

\medskip

\[
\begin{aligned}
&i=t+1,\quad j=k_r, \\[0.3em]
T\left((\mathcal{W}^s_{t+1}\mathcal{W}^r_{k_r})_l;z\right)
&=
1-z^{k_s}-z^{k_r}-z^{k_s+k_r-t+l} \\
&=
1-z^{k_s}-z^{k_r}-z^m, \\[0.3em]
&|\mathcal{W}^s|\geq1,\quad |\mathcal{W}^r|\geq1.
\end{aligned}
\]
\end{minipage}
\hfill
\begin{minipage}[t]{0.48\textwidth}
\small
\noindent
$(s,r)$\textit{-optimization.}

\medskip

\[
s,r\in\{1,2\}, \quad s\neq r.
\]

\[
\begin{aligned}
T\left((\mathcal{W}^1_{t+1}\mathcal{W}^2_{k_2})_l;z\right)
&=
1-z^{k_1}-z^{k_2}-z^m, \\[0.3em]
&|\mathcal{W}^1|\geq1,\quad |\mathcal{W}^2|\geq1.
\end{aligned}
\]
\end{minipage}

\bigskip

A digraph of type
\begin{equation*}
    (\mathcal{W}^s_i\mathcal{H})_l,
    \qquad s\in\{1,2\}.
\end{equation*}

\begin{figure}[ht]
    \centering
    \includegraphics[width=0.5\textwidth]{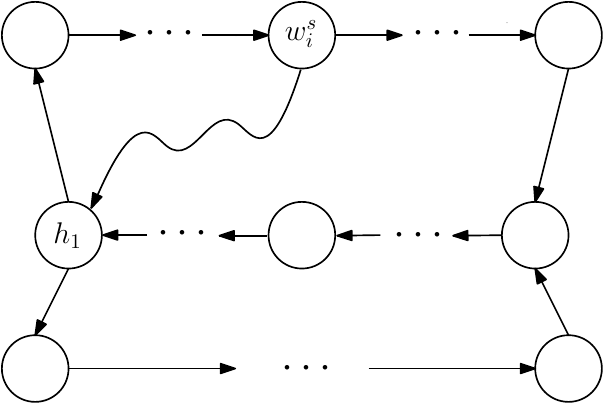}
    \caption{
    \((\mathcal{W}^s_i\mathcal{H})_l
    \text{-}\mathcal{B}_{k_s,k_r}^{\,t}\).}
    \label{fig:WH}
\end{figure}

\medskip

\noindent
\begin{minipage}[t]{0.4\textwidth}
\small
\textit{Recurrence relations.}
\[
\begin{aligned}
    H_1&=1+zW^s_{k_s}+zW^r_{k_r}, \\
    W^s_{k_s}&=\frac{1-z^{k_s-i}}{1-z}+z^{k_s-i}W^s_i, \\
    W^s_i&=1+zW^s_{i-1}+zE_1, \\
    E_1&=\frac{1-z^l}{1-z}+z^lH_1, \\
    W^s_{i-1}&=\frac{1-z^{i-2}}{1-z}+z^{i-2}H_1, \\
    W^r_{k_r}&=\frac{1-z^{k_r-1}}{1-z}+z^{k_r-1}H_1.
\end{aligned}
\]
\end{minipage}
\hfill
\begin{minipage}[t]{0.6\textwidth}
\small
\textit{Solution for $H_1$.}
\begin{equation*}
H_1
=
\frac{
1+z-z^{k_s}-z^{k_r}
+z^{k_s-i+2}
-z^{k_s-i+l+2}
}{
(1-z)\left(
1-z^{k_s}-z^{k_r}
-z^{k_s-i+l+2}
\right)
}.
\end{equation*}

\vspace{1em}

\textit{Topological polynomial.}
\begin{equation*}
    T\left((\mathcal{W}^s_i\mathcal{H})_l;z\right)
    =
    1-z^{k_s}-z^{k_r}
    -z^{k_s-i+l+2}.
\end{equation*}
\end{minipage}

\medskip

In this case, we have to derive the optimal choice of the parameter $i$ in
more detail. If $|\mathcal{B}|\geq 1$ or $l\neq 0$, then the optimal choice is $i=t+1$.
On the other hand, if $|\mathcal{B}|=0$ and $l=0$,
then we have to shift the index to $i=t+2$.
Indeed, in this case the edge $(w^s_{t+1},h_1)$ already belongs to the
original digraph $\mathcal{B}_{k_s,k_r}^{\,t}$. Consequently, the
realizability condition also becomes stronger, namely
\[
    |\mathcal{W}^s|\geq 2.
\]   
Moreover, if $ |\mathcal{B}|=0 $, then necessarily $t=1$. To avoid unnecessarily lengthening the text, we directly make the optimal choice
\[
s=1
\qquad\text{and}\qquad
r=2.
\]
\medskip

\noindent
$(s,r,i,j)$\textit{-optimization.}

\medskip

\noindent
\begin{minipage}[t]{0.48\textwidth}
\small
\noindent
$|\mathcal{B}|\geq1 \text{ or }l\neq0:$

\[
\begin{aligned}
&i=t+1,\quad s=1,\quad r=2, \\[0.3em]
T\left((\mathcal{W}^1_{t+1}\mathcal{H})_l;z\right)
&=
1-z^{k_1}-z^{k_2}-z^{k_1-t+l+1},\\[0.3em]
&|\mathcal{W}^1|\geq 1,\quad |\mathcal{B}|\geq1 \text{ or } l\neq0.
\end{aligned}
\]
\end{minipage}
\hfill
\begin{minipage}[t]{0.48\textwidth}
\small
\noindent
$|\mathcal{B}|=0$ and $l=0:$

\[
\begin{aligned}
&i=t+2,\quad s=1,\quad r=2,\\[0.3em]
T\left((\mathcal{W}^1_{t+2}\mathcal{H})_0;z\right)
&=
1-z^{k_1}-z^{k_2}-z^{k_1-1},\\[0.3em]
&|\mathcal{W}^1|\geq 2,\quad |\mathcal{B}|=0, \quad l=0.
\end{aligned}
\]
\end{minipage}

\medskip

Next, consider the type
\begin{equation*}
    (\mathcal{B}_i\mathcal{W}^s_j)_l,
    \qquad s\in\{1,2\}.
\end{equation*}

\begin{figure}[ht]
    \centering
    \includegraphics[width=0.5\textwidth]{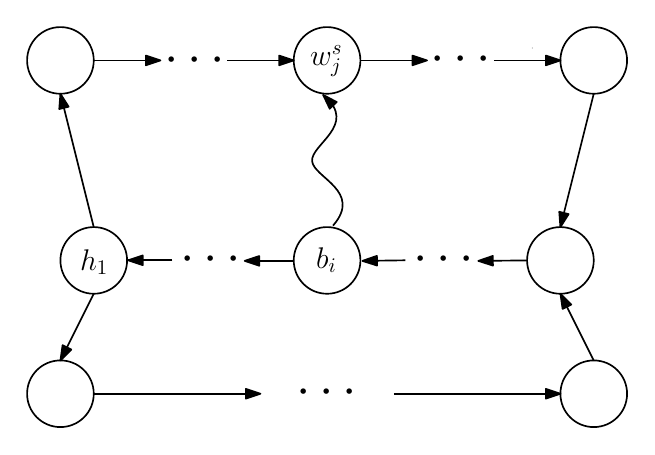}
    \caption{$(\mathcal{B}_i\mathcal{W}^s_j)_l
    \text{-}\mathcal{B}_{k_s,k_r}^{\,t}$.}
    \label{fig:BW}
\end{figure}

\medskip

\noindent
\begin{minipage}[t]{0.4\textwidth}
\small
\textit{Recurrence relations.}
\[
\begin{aligned}
    H_1&=1+zW^s_{k_s}+zW^r_{k_r}, \\
    W^s_{k_s}&=\frac{1-z^{k_s-i}}{1-z}+z^{k_s-i}B_i, \\
    B_i&=1+zB_{i-1}+zE_1, \\
    E_1&=\frac{1-z^l}{1-z}+z^lW^s_j, \\
    W_j^s&=\frac{1-z^{j-i}}{1-z}+z^{j-i}B_i, \\
    B_{i-1}&=\frac{1-z^{i-2}}{1-z}+z^{i-2}H_1, \\
    W^r_{k_r}&=\frac{1-z^{k_r-i}}{1-z}+z^{k_r-i}B_i.
\end{aligned}
\]
\end{minipage}
\hfill
\begin{minipage}[t]{0.6\textwidth}
\small
\textit{Solution for $H_1$.}
\begin{equation*}
H_1
=
\frac{
Q(z)
}{
(1-z)
\left(
1-z^{k_s}-z^{k_r}
-z^{l+j-i+1}
\right)
}.
\end{equation*}

\vspace{1em}

\textit{Topological polynomial.}
\begin{equation*}
    T\left((\mathcal{B}_i\mathcal{W}^s_j)_l;z\right)
    =
    1-z^{k_s}-z^{k_r}
    -z^{l+j-i+1}.
\end{equation*}
\end{minipage}

\medskip

\noindent
where
\begin{equation*}
    Q(z)
    =
    1+z
    -z^{l+j-i+1}
    -z^{l+j-i+2}
    +z^{k_s-i+2}
    +z^{k_r-i+2}
    -z^{k_s}
    -z^{k_r}.
\end{equation*}

\medskip

\noindent
\begin{minipage}[t]{0.48\textwidth}
\small
\noindent
$(i,j)$\textit{-optimization.}

\medskip

\[
\begin{aligned}
&i=2,\quad j=k_s, \\[0.3em]
T\left((\mathcal{B}_2\mathcal{W}^s_{k_s})_l;z\right)
&=
1-z^{k_s}-z^{k_r}
-z^{k_s+l-1}, \\[0.3em]
&|\mathcal{W}^s|\geq1,\quad |\mathcal{B}|\geq1.
\end{aligned}
\]
\end{minipage}
\hfill
\begin{minipage}[t]{0.48\textwidth}
\small
\noindent
$(s,r)$\textit{-optimization.}

\medskip

\[
\begin{aligned}
&s=1,\quad r=2, \\
T\left((\mathcal{B}_2\mathcal{W}^1_{k_1})_l;z\right)
&=
1-z^{k_1}-z^{k_2}
-z^{k_1+l-1}, \\[0.3em]
&|\mathcal{W}^1|\geq1,\quad |\mathcal{B}|\geq1.
\end{aligned}
\]
\end{minipage}

\medskip

Next, consider the type
\begin{equation*}
    (\mathcal{B}_i\mathcal{H})_l.
\end{equation*}

\begin{figure}[ht]
    \centering
    \includegraphics[width=0.5\textwidth]{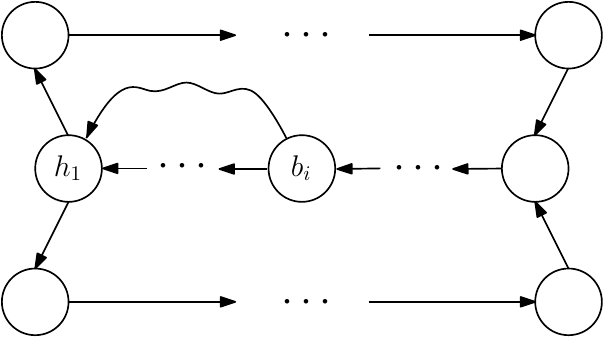}
    \caption{$(\mathcal{B}_i\mathcal{H})_l
    \text{-}\mathcal{B}_{k_s,k_r}^{\,t}$.}
    \label{fig:BH}
\end{figure}

\medskip

\noindent
\begin{minipage}[t]{0.4\textwidth}
\small
\textit{Recurrence relations.}
\[
\begin{aligned}
    H_1&=1+zW^s_{k_s}+zW^r_{k_r}, \\
    W^s_{k_s}&=\frac{1-z^{k_s-i}}{1-z}+z^{k_s-i}B_i, \\
    W^r_{k_r}&=\frac{1-z^{k_r-i}}{1-z}+z^{k_r-i}B_i, \\
    B_i&=1+zE_1+zB_{i-1}, \\
    E_1&=\frac{1-z^l}{1-z}+z^lH_1, \\
    B_{i-1}&=\frac{1-z^{i-2}}{1-z}+z^{i-2}H_1.
\end{aligned}
\]
\end{minipage}
\hfill
\begin{minipage}[t]{0.6\textwidth}
\small
\textit{Solution for $H_1$.}
\begin{equation*}
H_1
=
\frac{
Q(z)
}{
(1-z)\left(
1-z^{k_s}-z^{k_r}
-z^{k_s-i+l+2}
-z^{k_r-i+l+2}
\right)
}.
\end{equation*}

\vspace{1em}

\textit{Topological polynomial.}
\begin{equation*}
    T\left((\mathcal{B}_i\mathcal{H})_l;z\right)
    =
    1-z^{k_s}-z^{k_r}
    -z^{k_s-i+l+2}
    -z^{k_r-i+l+2}.
\end{equation*}
\end{minipage}

\medskip

\noindent
where
\begin{equation*}
    Q(z)
    =
    1+z
    +z^{k_s-i+2}
    +z^{k_r-i+2}
    -z^{k_s}
    -z^{k_r}
    -z^{k_s-i+l+2}
    -z^{k_r-i+l+2}.
\end{equation*}

\medskip

\noindent
\begin{minipage}[t]{0.48\textwidth}
\small
\noindent
$(i)$\textit{-optimization.}

\medskip

\noindent
$l\neq 0:$
\[
\begin{aligned}
&i=2, \\[0.3em]
T\left((\mathcal{B}_2\mathcal{H})_l;z\right)
&=
1-z^{k_s}-z^{k_r}
-z^{k_s+l}
-z^{k_r+l}, \\[0.3em]
&|\mathcal{B}|\geq1.
\end{aligned}
\]

\noindent
$l=0:$
\[
\begin{aligned}
&i=3, \\[0.3em]
T\left((\mathcal{B}_3\mathcal{H})_0;z\right)
&=
1-z^{k_s}-z^{k_r}
-z^{k_s-1}
-z^{k_r-1}, \\[0.3em]
&|\mathcal{B}|\geq2.
\end{aligned}
\]
\end{minipage}
\hfill
\begin{minipage}[t]{0.48\textwidth}
\small
\noindent
$(s,r)$\textit{-optimization.}

\medskip

\[
s,r\in \{1,2\}, \quad s\neq r.
\]

\noindent
$l\neq 0:$
\[
\begin{aligned}
T\left((\mathcal{B}_2\mathcal{H})_l;z\right)
&=
1-z^{k_1}-z^{k_2}
-z^{k_1+l}
-z^{k_2+l}, \\[0.3em]
&|\mathcal{B}|\geq1.
\end{aligned}
\]

\noindent
$l=0:$
\[
\begin{aligned}
T\left((\mathcal{B}_3\mathcal{H})_0;z\right)
&=
1-z^{k_1}-z^{k_2}
-z^{k_1-1}
-z^{k_2-1}, \\[0.3em]
&|\mathcal{B}|\geq2.
\end{aligned}
\]
\end{minipage}

\medskip

\begin{lemma}\label{lem:BH_l0}
    For $l\geq1$,
    \[
    (\mathcal{B}_2\mathcal{H})_l\text{-}\mathcal{B}_{k_1,k_2}^{\,t}\cong(\mathcal{B}_{l+2}\mathcal{H})_0\text{-}\mathcal{B}_{k_1+l,k_2+l}^{\,t+l}.
    \]
\end{lemma}

\begin{proof}
The proof is analogous to that of Lemma~\ref{lem:WB_l=0}. We only give the corresponding isomorphism.

Let
\[
G=(\mathcal{B}_2\mathcal{H})_l\text{-}\mathcal{B}_{k_1,k_2}^{\,t}
\]
and
\[
H=(\mathcal{B}_{l+2}\mathcal{H})_0\text{-}\mathcal{B}_{k_1+l,k_2+l}^{\,t+l}.
\]
Define
\[
\phi:V(G)\to V(H)
\]
by
\[
\phi(v)=
\begin{cases}
b_{l+2-\alpha}, & \text{if } v=e_\alpha,\ \alpha=1,\dots,l,\\
b_{l+\alpha}, & \text{if } v=b_\alpha,\ \alpha=2,\dots,t,\\
v, & \text{otherwise}.
\end{cases}
\]
As in Lemma~\ref{lem:WB_l=0}, this map is a digraph isomorphism. Hence $G\cong H$.
\end{proof}

\begin{corollary}
To find the minimum spectral radius in the class of digraphs
\[
(\mathcal{B}_i\mathcal{H})_l,
\]
it is sufficient to consider only the case $l=0$.
\end{corollary}

\begin{proof}
The proof is analogous to that of Corollary~\ref{cor:WB_l0}, using Lemma~\ref{lem:BH_l0}.
\end{proof}

Next, consider the type
\begin{equation*}
    (\mathcal{H}\mathcal{W}^s_j)_l,
    \qquad s\in\{1,2\}.
\end{equation*}

\begin{figure}[ht]
    \centering
    \includegraphics[width=0.5\textwidth]{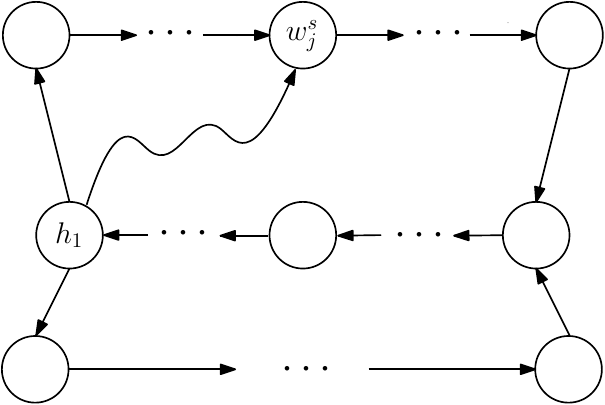}
    \caption{$(\mathcal{H}\mathcal{W}^s_j)_l
    \text{-}\mathcal{B}_{k_s,k_r}^{\,t}$.}
    \label{fig:HW}
\end{figure}

\medskip

\noindent
\begin{minipage}[t]{0.4\textwidth}
\small
\textit{Recurrence relations.}
\[
\begin{aligned}
    H_1&=1+zW^s_{k_s}+zW^r_{k_r}+zE_1, \\
    W^s_{k_s}&=\frac{1-z^{k_s-1}}{1-z}+z^{k_s-1}H_1, \\
    W^r_{k_r}&=\frac{1-z^{k_r-1}}{1-z}+z^{k_r-1}H_1, \\
    E_1&=\frac{1-z^l}{1-z}+z^lW^s_j, \\
    W^s_j&=\frac{1-z^{j-1}}{1-z}+z^{j-1}H_1.
\end{aligned}
\]
\end{minipage}
\hfill
\begin{minipage}[t]{0.6\textwidth}
\small
\textit{Solution for $H_1$.}
\begin{equation*}
H_1
=
\frac{
1+2z
-z^{k_s}
-z^{k_r}
-z^{l+j}
}{
(1-z)\left(
1-z^{k_s}
-z^{k_r}
-z^{l+j}
\right)
}.
\end{equation*}

\vspace{1em}

\textit{Topological polynomial.}
\begin{equation*}
    T\left((\mathcal{H}\mathcal{W}^s_j)_l;z\right)
    =
    1-z^{k_s}
    -z^{k_r}
    -z^{l+j}.
\end{equation*}
\end{minipage}

\medskip

\noindent
\begin{minipage}[t]{0.48\textwidth}
\small
\noindent
$(j)$\textit{-optimization.}

\medskip

\noindent
$l\neq 0:$
\[
\begin{aligned}
&j=k_s, \\[0.3em]
T\left((\mathcal{H}\mathcal{W}^s_{k_s})_l;z\right)
&=
1-z^{k_s}
-z^{k_r}
-z^{k_s+l}, \\[0.3em]
&|\mathcal{W}^s|\geq1.
\end{aligned}
\]

\noindent
$l=0:$
\[
\begin{aligned}
&j=k_s-1, \\[0.3em]
T\left((\mathcal{H}\mathcal{W}^s_{k_s-1})_0;z\right)
&=
1-z^{k_s}
-z^{k_r}
-z^{k_s-1}, \\[0.3em]
&|\mathcal{W}^s|\geq2.
\end{aligned}
\]
\end{minipage}
\hfill
\begin{minipage}[t]{0.48\textwidth}
\small
\noindent
$(s,r)$\textit{-optimization.}

\medskip

\[
s=1,\quad r=2:
\]

\noindent
$l\neq 0:$
\[
\begin{aligned}
T\left((\mathcal{H}\mathcal{W}^1_{k_1})_l;z\right)
&=
1-z^{k_1}
-z^{k_2}
-z^{k_1+l}, \\[0.3em]
&|\mathcal{W}^1|\geq1.
\end{aligned}
\]

\noindent
$l=0:$
\[
\begin{aligned}
T\left((\mathcal{H}\mathcal{W}^1_{k_1-1})_0;z\right)
&=
1-z^{k_1}
-z^{k_2}
-z^{k_1-1}, \\[0.3em]
&|\mathcal{W}^1|\geq2.
\end{aligned}
\]
\end{minipage}

\begin{lemma}\label{lem:HW_l=0}
    For $l\geq1$,
    \[
    (\mathcal{H}\mathcal{W}^1_{k_1})_l\text{-}\mathcal{B}^{\,t}_{k_1,k_2}\cong (\mathcal{H}\mathcal{W}^1_{k_1})_0\text{-}\mathcal{B}^{\,t}_{k_1+l,k_2}.
    \]
\end{lemma}

\begin{proof}
The proof is analogous to that of Lemma~\ref{lem:WB_l=0}. We only give the corresponding isomorphism.

Let
\[
G=(\mathcal{H}\mathcal{W}^1_{k_1})_l\text{-}\mathcal{B}^{\,t}_{k_1,k_2}
\]
and
\[
H=(\mathcal{H}\mathcal{W}^1_{k_1})_0\text{-}\mathcal{B}^{\,t}_{k_1+l,k_2}.
\]
Define
\[
\phi:V(G)\to V(H)
\]
by
\[
\phi(v)=
\begin{cases}
w^1_{k_1+l+1-\alpha}, & \text{if } v=e_\alpha,\ \alpha=1,\dots,l,\\
v, & \text{otherwise}.
\end{cases}
\]
As in Lemma~\ref{lem:WB_l=0}, this map is a bijection and preserves directed edges. Hence $\phi$ is a digraph isomorphism, and therefore $G\cong H$.
\end{proof}

\begin{corollary}
To find the minimum spectral radius in the class of digraphs
\[
(\mathcal{H}\mathcal{W}^s_j)_l,
\]
it is sufficient to consider only the case $l=0$.
\end{corollary}

\begin{proof}
The proof is analogous to that of Corollary~\ref{cor:WB_l0}, using Lemma~\ref{lem:HW_l=0}.
\end{proof}

\medskip

It remains to consider the last cross-region type, namely
\begin{equation*}
    (\mathcal{H}\mathcal{B}_j)_l.
\end{equation*}

\begin{figure}[ht]
    \centering
    \includegraphics[width=0.5\textwidth]{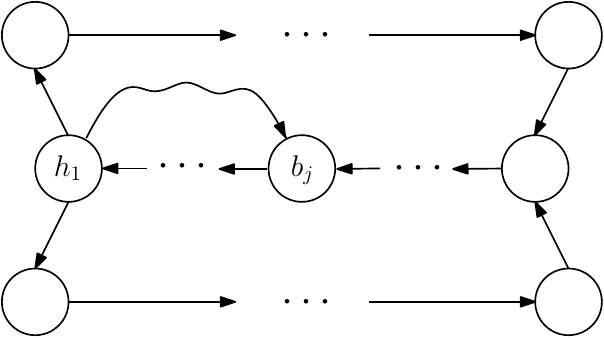}
    \caption{$(\mathcal{H}\mathcal{B}_j)_l
    \text{-}\mathcal{B}_{k_s,k_r}^{\,t}$.}
    \label{fig:HB}
\end{figure}

\medskip

\noindent
\begin{minipage}[t]{0.4\textwidth}
\small
\textit{Recurrence relations.}
\[
\begin{aligned}
    H_1&=1+zW^s_{k_s}+zW^r_{k_r}+zE_1, \\
    W^s_{k_s}&=\frac{1-z^{k_s-1}}{1-z}+z^{k_s-1}H_1, \\
    W^r_{k_r}&=\frac{1-z^{k_r-1}}{1-z}+z^{k_r-1}H_1, \\
    E_1&=\frac{1-z^l}{1-z}+z^lB_j, \\
    B_j&=\frac{1-z^{j-1}}{1-z}+z^{j-1}H_1.
\end{aligned}
\]
\end{minipage}
\hfill
\begin{minipage}[t]{0.6\textwidth}
\small
\textit{Solution for $H_1$.}
\begin{equation*}
H_1
=
\frac{
1+2z
-z^{k_s}
-z^{k_r}
-z^{l+j}
}{
(1-z)\left(
1-z^{k_s}
-z^{k_r}
-z^{l+j}
\right)
}.
\end{equation*}

\vspace{1em}

\textit{Topological polynomial.}
\begin{equation*}
    T\left((\mathcal{H}\mathcal{B}_j)_l;z\right)
    =
    1-z^{k_s}
    -z^{k_r}
    -z^{l+j}.
\end{equation*}
\end{minipage}

\medskip

\noindent
\begin{minipage}[t]{0.48\textwidth}
\small
\noindent
$(j)$\textit{-optimization.}

\medskip

\noindent
$l\neq 0 \text{ or } |\mathcal{W}^2|\neq 0:$
\[
\begin{aligned}
&j=t, \\[0.3em]
T\left((\mathcal{H}\mathcal{B}_t)_l;z\right)
&=
1-z^{k_s}
-z^{k_r}
-z^{l+t}, \\[0.3em]
&|\mathcal{B}|\geq1.
\end{aligned}
\]

\noindent
$l=0,\ |\mathcal{W}^2|=0:$
\[
\begin{aligned}
&j=t-1, \\[0.3em]
T\left((\mathcal{H}\mathcal{B}_{t-1})_0;z\right)
&=
1-z^{k_s}
-z^{k_r}
-z^{t-1}, \\[0.3em]
&|\mathcal{B}|\geq2.
\end{aligned}
\]
\end{minipage}
\hfill
\begin{minipage}[t]{0.48\textwidth}
\small
\noindent
$(s,r)$\textit{-optimization.}

\medskip

\[
s,r\in\{1,2\}, \quad s\neq r.
\]

\noindent
$l\neq 0 \text{ or } |\mathcal{W}^2|\neq 0:$
\[
\begin{aligned}
T\left((\mathcal{H}\mathcal{B}_t)_l;z\right)
&=
1-z^{k_1}
-z^{k_2}
-z^{l+t}, \\[0.3em]
&|\mathcal{B}|\geq1.
\end{aligned}
\]

\noindent
$l=0,\ |\mathcal{W}^2|=0:$
\[
\begin{aligned}
T\left((\mathcal{H}\mathcal{B}_{t-1})_0;z\right)
&=
1-z^{k_1}
-z^{k_2}
-z^{t-1}, \\[0.3em]
&|\mathcal{B}|\geq2.
\end{aligned}
\]
\end{minipage}

\medskip

\subsection{Intra-region types}

Consider the type
\[
(\mathcal{W}^s_i\mathcal{W}^s_j)_l^+, \quad s\in\{1,2\}.
\]

\begin{figure}[H]
    \centering
    \includegraphics[width=0.55\textwidth]{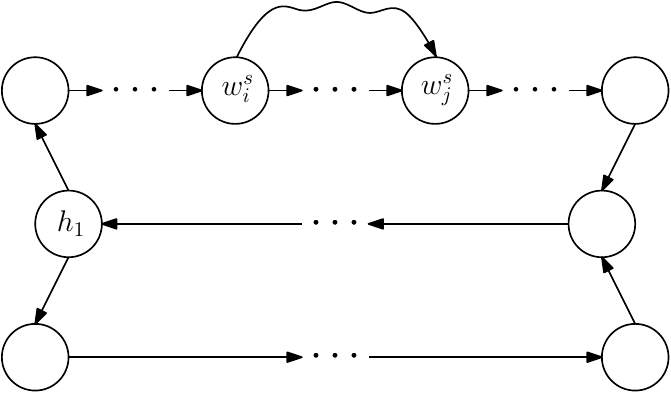}
    \caption{$(\mathcal{W}^s_i\mathcal{W}^s_j)_l^+
    \text{-}\mathcal{B}_{k_s,k_r}^{\,t}$.}
    \label{fig:WW+}
\end{figure}

\medskip

\noindent
\begin{minipage}[t]{0.4\textwidth}
\small
\textit{Recurrence relations.}
\[
\begin{aligned}
    H_1&=1+zW^s_{k_s}+zW^r_{k_r}, \\
    W^s_{k_s}&=\frac{1-z^{k_s-i}}{1-z}+z^{k_s-i}W^s_i, \\
    W^s_i&=1+zW^s_{i-1}+zE_1, \\
    W^s_{i-1}&=\frac{1-z^{i-2}}{1-z}+z^{i-2}H_1, \\
    E_1&=\frac{1-z^l}{1-z}+z^lW^s_j, \\
    W^s_j&=\frac{1-z^{j-1}}{1-z}+z^{j-1}H_1, \\
    W^r_{k_r}&=\frac{1-z^{k_r-1}}{1-z}+z^{k_r-1}H_1.
\end{aligned}
\]
\end{minipage}
\hfill
\begin{minipage}[t]{0.6\textwidth}
\small
\textit{Solution for $H_1$.}
\begin{equation*}
H_1
=
\frac{
1+z-z^{k_s}-z^{k_r}
+z^{k_s-i+2}
-z^{k_s-i+l+j+1}
}{
(1-z)\left(
1-z^{k_s}
-z^{k_r}
-z^{k_s-i+l+j+1}
\right)
}.
\end{equation*}

\vspace{1em}

\textit{Topological polynomial.}
\begin{equation*}
    T\left((\mathcal{W}^s_i\mathcal{W}^s_j)^+_l;z\right)
    =
    1-z^{k_s}
    -z^{k_r}
    -z^{k_s-i+l+j+1}.
\end{equation*}
\end{minipage}

\noindent
\begin{minipage}[t]{0.48\textwidth}
\small
\noindent
$(i,j)$\textit{-optimization.}

\medskip

\noindent
$l\neq 0:$
\[
\begin{aligned}
&j=i-1, \\[0.3em]
T\left((\mathcal{W}^s_i\mathcal{W}^s_{i-1})^+_l;z\right)
&=
1-z^{k_s}
-z^{k_r}
-z^{k_s+l}, \\[0.3em]
&|\mathcal{W}^s|\geq2.
\end{aligned}
\]

\noindent
$l=0:$
\[
\begin{aligned}
&j=i-2, \\[0.3em]
T\left((\mathcal{W}^s_i\mathcal{W}^s_{i-2})^+_0;z\right)
&=
1-z^{k_s}
-z^{k_r}
-z^{k_s-1}, \\[0.3em]
&|\mathcal{W}^s|\geq3.
\end{aligned}
\]
\end{minipage}
\hfill
\begin{minipage}[t]{0.48\textwidth}
\small
\noindent
$(s,r)$\textit{-optimization.}

\medskip

\noindent
$l\neq 0:$
\[
\begin{aligned}
&s=1,\quad r=2, \\[0.3em]
T\left((\mathcal{W}^1_i\mathcal{W}^1_{i-1})^+_l;z\right)
&=
1-z^{k_1}
-z^{k_2}
-z^{k_1+l}, \\[0.3em]
&|\mathcal{W}^1|\geq2.
\end{aligned}
\]

\noindent
$l=0:$
\[
\begin{aligned}
&s=1,\quad r=2, \\[0.3em]
T\left((\mathcal{W}^1_i\mathcal{W}^1_{i-2})^+_0;z\right)
&=
1-z^{k_1}
-z^{k_2}
-z^{k_1-1}, \\[0.3em]
&|\mathcal{W}^1|\geq3.
\end{aligned}
\]
\end{minipage}

\begin{lemma}\label{lem:WW_l=0}
    For $l\geq1$,
    \[
    (\mathcal{W}^1_i\mathcal{W}^1_{i-1})^+_l\text{-}\mathcal{B}^{\,t}_{k_1,k_2}\cong
    (\mathcal{W}^1_{i+l}\mathcal{W}^1_{i-1})^+_0\text{-}\mathcal{B}^{\,t}_{k_1+l,k_2}.
    \]
\end{lemma}

\begin{proof}
The proof is analogous to that of Lemma~\ref{lem:WB_l=0}. We only give the corresponding isomorphism.

Let
\[
G=(\mathcal{W}^1_i\mathcal{W}^1_{i-1})^+_l\text{-}\mathcal{B}^{\,t}_{k_1,k_2}
\]
and
\[
H=(\mathcal{W}^1_{i+l}\mathcal{W}^1_{i-1})^+_0\text{-}\mathcal{B}^{\,t}_{k_1+l,k_2}.
\]
Define
\[
\phi:V(G)\to V(H)
\]
by
\[
\phi(v)=
\begin{cases}
w^1_{l+\alpha}, & \text{if } v=w^1_\alpha,\ \alpha=i,\dots,k_1, \\
w^1_{i+l-\alpha}, & \text{if } v=e_\alpha,\ \alpha=1,\dots,l, \\
v, & \text{otherwise}.
\end{cases}
\]
As in Lemma~\ref{lem:WB_l=0}, this map is a digraph isomorphism. Hence $G\cong H$.
\end{proof}

\begin{corollary}
To find the minimum spectral radius in the class of digraphs
\[
(\mathcal{W}^s_i\mathcal{W}^s_j)_l^+,
\]
it is sufficient to consider only the case $l=0$.
\end{corollary}

\begin{proof}
The proof is analogous to that of Corollary~\ref{cor:WB_l0}, using Lemma~\ref{lem:WW_l=0}.
\end{proof}

\medskip

The type
\[
(\mathcal{W}^s_i\mathcal{W}^s_j)_l^-, \quad s\in\{1,2\}.
\]

\begin{figure}[H]
    \centering
    \includegraphics[width=0.55\textwidth]{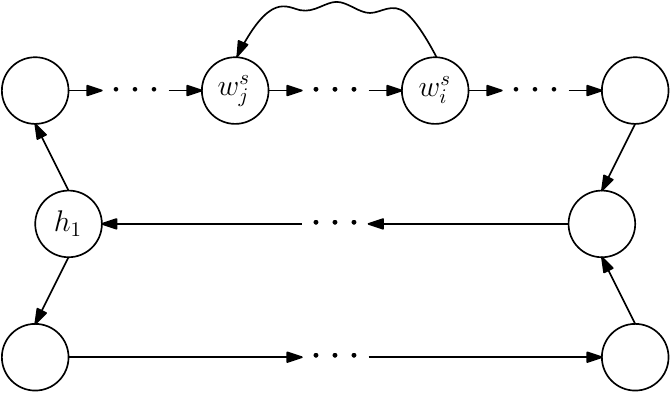}
    \caption{$(\mathcal{W}^s_i\mathcal{W}^s_j)_l^-
    \text{-}\mathcal{B}_{k_s,k_r}^{\,t}$.}
    \label{fig:WW-}
\end{figure}

\medskip

\noindent
\begin{minipage}[t]{0.4\textwidth}
\small
\textit{Recurrence relations.}
\[
\begin{aligned}
    H_1&=1+zW^s_{k_s}+zW^r_{k_r}, \\
    W^s_{k_s}&=\frac{1-z^{k_s-i}}{1-z}+z^{k_s-i}W^s_i, \\
    W^s_i&=1+zW^s_{i-1}+zE_1, \\
    E_1&=\frac{1-z^l}{1-z}+z^lW^s_j, \\
    W^s_j&=\frac{1-z^{j-i}}{1-z}+z^{j-i}W^s_i, \\
    W^s_{i-1}&=\frac{1-z^{i-2}}{1-z}+z^{i-2}H_1, \\
    W^r_{k_r}&=\frac{1-z^{k_r-1}}{1-z}+z^{k_r-1}H_1.
\end{aligned}
\]
\end{minipage}
\hfill
\begin{minipage}[t]{0.6\textwidth}
\small
\textit{Solution for $H_1$.}
\begin{equation*}
H_1
=
\frac{
Q(z)
}{
(1-z)\left(
1-z^{k_s}
-z^{k_r}
-z^{l+j-i+1}
+z^{k_r+l+j-i+1}
\right)
}.
\end{equation*}

\vspace{1em}

\textit{Topological polynomial.}
\begin{equation*}
    T\left((\mathcal{W}^s_i\mathcal{W}^s_j)^-_l;z\right)
    =
    1-z^{k_s}
    -z^{k_r}
    -z^{l+j-i+1}
    +z^{k_r+l+j-i+1}.
\end{equation*}
\end{minipage}

\medskip
\noindent
where
\[
Q(z)=1+z
-z^{l+j-i+1}
-z^{l+j-i+2}
+z^{k_s-i+2}
+z^{k_r+l+j-i+1}
-z^{k_s}
-z^{k_r}.
\]

We show that the polynomial
\[
T\left((\mathcal{W}^s_i\mathcal{W}^s_j)^-_l;z\right)
=
1-z^{k_s}
-z^{k_r}
-z^{l+j-i+1}
+z^{k_r+l+j-i+1}
\]
is strictly decreasing on the interval $(0,1)$ and takes the values $1$ and $-1$ at $z=0$ and $z=1$, respectively. Therefore, it satisfies the assumptions of Lemma~\ref{lem:topological_equation_from_vertex_gf}.

Since $k_s,k_r\geq 1$, $l\geq 0$, and
\[
j\in\{k_s,k_s-1,\dots,t+2\}
\quad \text{and} \quad
i\in\{j-1,j-2,\dots,t+1\},
\]
all exponents appearing in
\[
T\left((\mathcal{W}^s_i\mathcal{W}^s_j)^-_l;z\right)
\]
are at least $1$.

Differentiating, we obtain
\begin{align*}
T'\left((\mathcal{W}^s_i\mathcal{W}^s_j)^-_l;z\right)
&=
-k_s z^{k_s-1}
-k_r z^{k_r-1} 
-(l+j-i+1)z^{l+j-i} \\
&\quad\;
+(k_r+l+j-i+1)z^{k_r+l+j-i}
\\
&=
-k_s z^{k_s-1}
-k_r z^{k_r-1}\left(1-z^{l+j-i+1}\right)
-(l+j-i+1)z^{l+j-i}\left(1-z^{k_r}\right).
\end{align*}
For every $z\in(0,1)$, all three terms in the last expression are strictly negative. Hence
\[
T'\left((\mathcal{W}^s_i\mathcal{W}^s_j)^-_l;z\right)<0,
\]
and therefore
\[
T\left((\mathcal{W}^s_i\mathcal{W}^s_j)^-_l;z\right)
\]
is strictly decreasing on $(0,1)$. Moreover,
\[
T\left((\mathcal{W}^s_i\mathcal{W}^s_j)^-_l;0\right)=1
\quad \text{and} \quad
T\left((\mathcal{W}^s_i\mathcal{W}^s_j)^-_l;1\right)=-1.
\]
Therefore, this polynomial satisfies the conditions of Lemma~\ref{lem:topological_equation_from_vertex_gf}.

In order to optimize the difference $j-i$ in the exponents of the terms of the
topological polynomial
\[
T\left((\mathcal{W}^s_i\mathcal{W}^s_j)^-_l;z\right),
\]
we rewrite it in the form
\[
1-z^{k_s}-z^{k_r}-z^{l+j-i+1}\left(1-z^{k_r}\right).
\]
Since $1-z^{k_r}>0$ for every $z\in(0,1)$, the largest possible root is
obtained by maximizing the difference $j-i$. Therefore, the optimal choice is $j=k_s$ and $i=t+1$.
Thus, we obtain the $(i,j)$-optimized topological polynomial
\begin{align*}
    T\left((\mathcal{W}^s_{t+1}\mathcal{W}^s_{k_s})^-_l;z\right)
    &=
    1-z^{k_s}-z^{k_r}-z^{k_s+l-t}+z^{k_s+k_r+l-t} \\
    &=
    1-z^{k_s}-z^{k_r}-z^{k_s+l-t}+z^m,
    \qquad |\mathcal{W}^s|\geq 2,
\end{align*}

By choosing $s=1$ and $r=2$ we obtain the $(s,r,i,j)$-optimized polynomial
\[
T\left((\mathcal{W}^1_{t+1}\mathcal{W}^1_{k_1})^-_l;z\right)
=
1-z^{k_1}-z^{k_2}-z^{k_1+l-t}+z^m, \qquad |\mathcal{W}^1|\geq 2.
\]
\medskip

Next, consider the type
\[
(\mathcal{B}_{i}\mathcal{B}_{j})^+_l.
\]

\begin{figure}[H]
    \centering
    \includegraphics[width=0.55\textwidth]{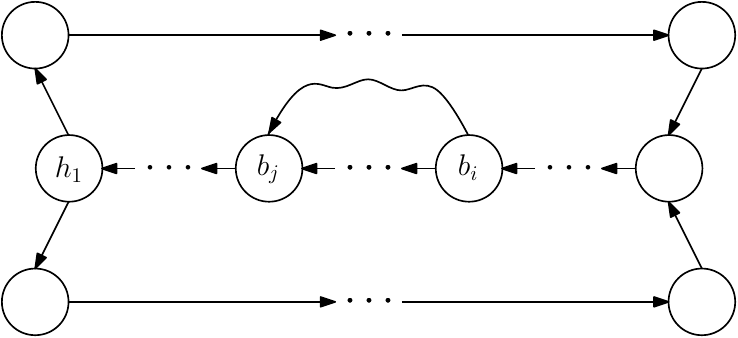}
    \caption{$(\mathcal{B}_{i}\mathcal{B}_{j})^+_l
    \text{-}\mathcal{B}_{k_s,k_r}^{\,t}$.}
    \label{fig:BB+}
\end{figure}

\medskip

\noindent
\begin{minipage}[t]{0.4\textwidth}
\small
\textit{Recurrence relations.}
\[
\begin{aligned}
    H_1&=1+zW^s_{k_s}+zW^r_{k_r}, \\
    W^s_{k_s}&=\frac{1-z^{k_s-i}}{1-z}+z^{k_s-i}B_i, \\
    W^r_{k_r}&=\frac{1-z^{k_r-i}}{1-z}+z^{k_r-i}B_i, \\
    B_i&=1+zE_1+zB_{i-1}, \\
    E_1&=\frac{1-z^l}{1-z}+z^lB_j, \\
    B_{i-1}&=\frac{1-z^{i-2}}{1-z}+z^{i-2}H_1, \\
    B_j&=\frac{1-z^{j-1}}{1-z}+z^{j-1}H_1.
\end{aligned}
\]
\end{minipage}
\hfill
\begin{minipage}[t]{0.6\textwidth}
\small
\textit{Solution for $H_1$.}
\begin{equation*}
H_1
=
\frac{
Q(z)
}{
(1-z)\left(
1-z^{k_s}
-z^{k_r}
-z^{k_s-i+l+j+1}
-z^{k_r-i+l+j+1}
\right)
}.
\end{equation*}

\vspace{1em}

\textit{Topological polynomial.}
\begin{equation*}
    T\left((\mathcal{B}_i\mathcal{B}_j)^+_l;z\right)
    =
    1-z^{k_s}
    -z^{k_r}
    -z^{k_s-i+l+j+1}
    -z^{k_r-i+l+j+1}.
\end{equation*}
\end{minipage}

\medskip
\noindent
where 
\[
Q(z)=1+z-z^{k_s}-z^{k_r}
+z^{k_s-i+2}
+z^{k_r-i+2}
-z^{k_s-i+l+j+1}
-z^{k_r-i+l+j+1}.
\]

\noindent
\begin{minipage}[t]{0.48\textwidth}
\small
\noindent
$(i,j)$\textit{-optimization.}

\medskip

\noindent
$l\neq 0:$
\[
\begin{aligned}
&j=i-1, \\[0.3em]
T\left((\mathcal{B}_i\mathcal{B}_{i-1})^+_l;z\right)
&=
1-z^{k_s}-z^{k_r} \\
&\quad
-z^{k_s+l}-z^{k_r+l}, \\[0.3em]
&|\mathcal{B}|\geq2.
\end{aligned}
\]

\noindent
$l=0:$
\[
\begin{aligned}
&j=i-2, \\[0.3em]
T\left((\mathcal{B}_i\mathcal{B}_{i-2})^+_0;z\right)
&=
1-z^{k_s}
-z^{k_r} \\
&\quad
-z^{k_s-1}
-z^{k_r-1}, \\[0.3em]
&|\mathcal{B}|\geq3.
\end{aligned}
\]
\end{minipage}
\hfill
\begin{minipage}[t]{0.48\textwidth}
\small
\noindent
$(s,r)$\textit{-optimization.}

\medskip

\noindent
$l\neq 0:$
\[
\begin{aligned}
&\hspace{-4em}s,r\in\{1,2\},\quad s\neq r, \\[0.3em]
T\left((\mathcal{B}_i\mathcal{B}_{i-1})^+_l;z\right)
&=
1-z^{k_1}-z^{k_2} \\
&\quad
-z^{k_1+l}-z^{k_2+l}, \\[0.3em]
&|\mathcal{B}|\geq2.
\end{aligned}
\]

\noindent
$l=0:$
\[
\begin{aligned}
&\hspace{-4em}s,r\in\{1,2\},\quad s\neq r, \\[0.3em]
T\left((\mathcal{B}_i\mathcal{B}_{i-2})^+_0;z\right)
&=
1-z^{k_1}
-z^{k_2} \\
&\quad
-z^{k_1-1}
-z^{k_2-1}, \\[0.3em]
&|\mathcal{B}|\geq3.
\end{aligned}
\]
\end{minipage}

\begin{lemma}\label{lem:BB_l=0}
    For $l\geq1$,
    \[
    (\mathcal{B}_{i}\mathcal{B}_{i-1})^+_l\text{-}\mathcal{B}^{\,t}_{k_1,k_2}\cong
    (\mathcal{B}_{i+l}\mathcal{B}_{i-1})^+_0\text{-}\mathcal{B}^{\,t+l}_{k_1+l,k_2+l}.
    \]
\end{lemma}

\begin{proof}
The proof is analogous to that of Lemma~\ref{lem:WB_l=0}. We only give the corresponding isomorphism.

Let
\[
G=(\mathcal{B}_{i}\mathcal{B}_{i-1})^+_l\text{-}\mathcal{B}^{\,t}_{k_1,k_2}
\]
and
\[
H=(\mathcal{B}_{i+l}\mathcal{B}_{i-1})^+_0\text{-}\mathcal{B}^{\,t+l}_{k_1+l,k_2+l}.
\]
Define
\[
\phi:V(G)\to V(H)
\]
by
\[
\phi(v)=
\begin{cases}
b_{l+\alpha}, & \text{if } v=b_\alpha,\ \alpha=i,\dots,t, \\
b_{i+l-\alpha}, & \text{if } v=e_\alpha,\ \alpha=1,\dots,l, \\
v, & \text{otherwise}.
\end{cases}
\]
As in Lemma~\ref{lem:WB_l=0}, this map is a digraph isomorphism. Hence $G\cong H$.
\end{proof}

\begin{corollary}
To find the minimum spectral radius in the class of digraphs
\[
(\mathcal{B}_i\mathcal{B}_j)_l^+,
\]
it is sufficient to consider only the case $l=0$.
\end{corollary}

\begin{proof}
The proof is analogous to that of Corollary~\ref{cor:WB_l0}, using Lemma~\ref{lem:BB_l=0}.
\end{proof}

\medskip

Next, consider the type
\[
(\mathcal{B}_{i}\mathcal{B}_{j})^-_l.
\]

\begin{figure}[H]
    \centering
    \includegraphics[width=0.55\textwidth]{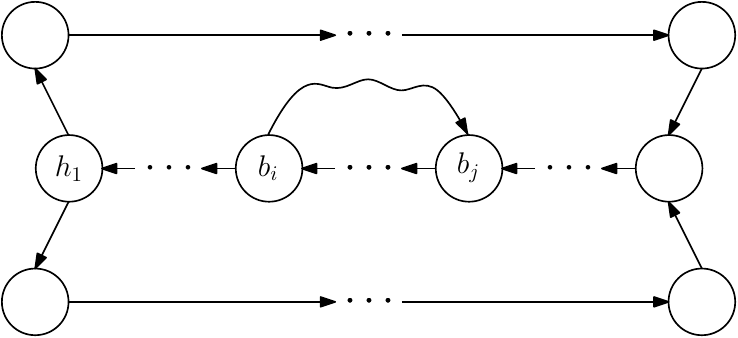}
    \caption{$(\mathcal{B}_{i}\mathcal{B}_{j})^-_l
    \text{-}\mathcal{B}_{k_s,k_r}^{\,t}$.}
    \label{fig:BB-}
\end{figure}

\medskip

\noindent
\begin{minipage}[t]{0.4\textwidth}
\small
\textit{Recurrence relations.}
\[
\begin{aligned}
    H_1&=1+zW^s_{k_s}+zW^r_{k_r}, \\
    W^s_{k_s}&=\frac{1-z^{k_s-i}}{1-z}+z^{k_s-i}B_i, \\
    W^r_{k_r}&=\frac{1-z^{k_r-i}}{1-z}+z^{k_r-i}B_i, \\
    B_i&=1+zB_{i-1}+zE_1, \\
    B_{i-1}&=\frac{1-z^{i-2}}{1-z}+z^{i-2}H_1, \\
    E_1&=\frac{1-z^l}{1-z}+z^lB_j, \\
    B_j&=\frac{1-z^{j-i}}{1-z}+z^{j-i}B_i.
\end{aligned}
\]
\end{minipage}
\hfill
\begin{minipage}[t]{0.6\textwidth}
\small
\textit{Solution for $H_1$.}
\begin{equation*}
H_1
=
\frac{
Q(z)
}{
(1-z)\left(
1-z^{k_s}
-z^{k_r}
-z^{l+j-i+1}
\right)
}.
\end{equation*}

\vspace{1em}

\textit{Topological polynomial.}
\begin{equation*}
    T\left((\mathcal{B}_i\mathcal{B}_j)^-_l;z\right)
    =
    1-z^{k_s}
    -z^{k_r}
    -z^{l+j-i+1}.
\end{equation*}
\end{minipage}

\medskip
\noindent
where
\[
Q(z)=1+z
-z^{l+j-i+1}
-z^{l+j-i+2}
+z^{k_s-i+2}
+z^{k_r-i+2}
-z^{k_s}
-z^{k_r}.
\]

\noindent
\begin{minipage}[t]{0.48\textwidth}
\small
\noindent
$(i,j)$\textit{-optimization.}

\medskip

\[
\begin{aligned}
&i=2,\quad j=t, \\[0.3em]
T\left((\mathcal{B}_2\mathcal{B}_t)^-_l;z\right)
&=
1-z^{k_s}
-z^{k_r}
-z^{l+t-2+1} \\
&=
1-z^{k_s}
-z^{k_r}
-z^{l+t-1}=0, \\[0.3em]
&|\mathcal{B}|\geq2.
\end{aligned}
\]
\end{minipage}
\hfill
\begin{minipage}[t]{0.48\textwidth}
\small
\noindent
$(s,r)$\textit{-optimization.}

\medskip

\[
\begin{aligned}
&s,r\in\{1,2\},\quad s\neq r, \\[0.3em]
T\left((\mathcal{B}_2\mathcal{B}_t)^-_l;z\right)
&=
1-z^{k_1}
-z^{k_2}
-z^{l+t-1}=0, \\[0.3em]
&|\mathcal{B}|\geq2.
\end{aligned}
\]
\end{minipage}

\medskip

Next, consider the type
\[
(\mathcal{H}\mathcal{H})_l.
\]

\begin{figure}[H]
    \centering
    \includegraphics[width=0.5\textwidth]{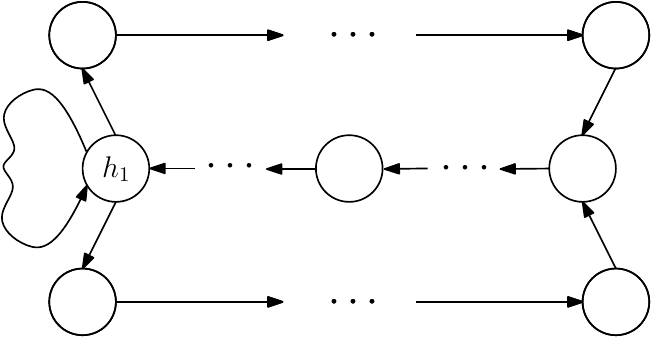}
    \caption{$(\mathcal{H}\mathcal{H})_l
    \text{-}\mathcal{B}_{k_s,k_r}^{\,t}$.}
    \label{fig:HH}
\end{figure}

\medskip

\noindent
\begin{minipage}[t]{0.4\textwidth}
\small
\textit{Recurrence relations.}
\[
\begin{aligned}
    H_1&=1+zW^s_{k_s}+zW^r_{k_r}+zE_1, \\
    W^s_{k_s}&=\frac{1-z^{k_s-1}}{1-z}+z^{k_s-1}H_1, \\
    W^r_{k_r}&=\frac{1-z^{k_r-1}}{1-z}+z^{k_r-1}H_1, \\
    E_1&=\frac{1-z^l}{1-z}+z^lH_1.
\end{aligned}
\]
\end{minipage}
\hfill
\begin{minipage}[t]{0.6\textwidth}
\small
\textit{Solution for $H_1$.}
\begin{equation*}
H_1
=
\frac{
1+2z
-z^{k_s}
-z^{k_r}
-z^{l+1}
}{
(1-z)\left(
1-z^{k_s}
-z^{k_r}
-z^{l+1}
\right)
}.
\end{equation*}

\vspace{1em}

\textit{Topological polynomial.}
\begin{equation*}
    T\left((\mathcal{H}\mathcal{H})_l;z\right)
    =
    1-z^{k_s}
    -z^{k_r}
    -z^{l+1}.
\end{equation*}
\end{minipage}

\medskip

In this case, no optimization of the parameters $s$ and $r$ is needed, since
the polynomial is symmetric in $k_s$ and $k_r$. Thus, we may write
$s=1$ and $r=2$.

The only point that requires additional attention is the case $l=0$. In this
case, the inserted ear is a loop at the vertex $h_1$. Since multiple edges
are not allowed, this loop is admissible if and only if the butterfly core
$\mathcal{B}^{\,t}_{k_1,k_2}$ does not already contain the loop $(h_1,h_1)$.

Under the assumption $k_1\geq k_2$, such a loop is already present precisely
when $k_2=t=1$. Since
\[
|\mathcal{W}^2|=k_2-t
\qquad\text{and}\qquad
|\mathcal{B}|=t-1,
\]
this excluded case is equivalent to
\[
|\mathcal{W}^2|=0
\qquad\text{and}\qquad
|\mathcal{B}|=0.
\]
Consequently, for $l=0$, the realizability condition is
\[
|\mathcal{W}^2|\geq 1
\qquad\text{or}\qquad
|\mathcal{B}|\geq 1.
\]

Thus, for $l\neq0$, we obtain
\[
T\left((\mathcal{H}\mathcal{H})_l;z\right)
=
1-z^{k_1}-z^{k_2}-z^{l+1}.
\]
For $l=0$, we obtain
\[
T\left((\mathcal{H}\mathcal{H})_0;z\right)
=
1-z^{k_1}-z^{k_2}-z,
\qquad
|\mathcal{W}^2|\geq1
\ \text{or}\
|\mathcal{B}|\geq1.
\]
\medskip

Finally, consider the two cycle types
\[
(\mathcal{W}^s_i\mathcal{W}^s_i)_l^0
\qquad\text{and}\qquad
(\mathcal{B}_i\mathcal{B}_i)_l^0.
\]
These two types are, in fact, already included in the types
\[
(\mathcal{W}^s_i\mathcal{W}^s_j)_l^-
\qquad\text{and}\qquad
(\mathcal{B}_i\mathcal{B}_j)_l^-,
\]
respectively. It is enough to substitute $j=i$ in the corresponding
recurrence relations.

For the type
\[
(\mathcal{W}^s_i\mathcal{W}^s_i)_l^0,
\]
the same argument as for
\[
(\mathcal{W}^s_i\mathcal{W}^s_j)_l^-
\]
remains valid after setting $j=i$. Hence, the corresponding polynomial
is strictly decreasing on $(0,1)$ and satisfies the conditions of
Lemma~\ref{lem:topological_equation_from_vertex_gf}.
For the type
\[
(\mathcal{B}_i\mathcal{B}_i)_l^0,
\]
the corresponding polynomial is clearly strictly decreasing on $(0,1)$,
with value $1$ at $z=0$ and a negative value at $z=1$, and hence it also
satisfies the conditions of
Lemma~\ref{lem:topological_equation_from_vertex_gf}.

Thus, we obtain the following $(s,r,i,j)$-optimized polynomials:
\[
T\left((\mathcal{W}^1_i\mathcal{W}^1_i)_l^0;z\right)
=
1-z^{k_1}-z^{k_2}-z^{l+1}+z^{k_2+l+1},
\qquad
|\mathcal{W}^1|\geq1,
\]
and
\[
T\left((\mathcal{B}_i\mathcal{B}_i)_l^0;z\right)
=
1-z^{k_1}-z^{k_2}-z^{l+1},
\qquad
|\mathcal{B}|\geq1.
\]

\section{Minimum spectral radius in the class $\mathcal{SC}_{m+2}(m)$}

\begin{table}[H]
\centering
\small
\setlength{\tabcolsep}{4pt}
\renewcommand{\arraystretch}{1.5}
\begin{tabular}{p{0.68\textwidth}|p{0.32\textwidth}}
\textbf{$(s,r,i,j)$-optimized topological polynomial}
&
\textbf{Realizability conditions}
\\
\hline
\hline

\(P_1(z):=T\left((\mathcal{W}^1_{t+2}\mathcal{B}_t)_0;z\right)
=1-z^{k_1}-z^{k_2}-z^{k_1-1}\)
&
$l=0,\ |\mathcal{W}^1|\geq2,\ |\mathcal{B}|\geq1$
\\
\hline

\(P_2(z):=T\left((\mathcal{W}^1_{t+1}\mathcal{B}_{t-1})_0;z\right)
=1-z^{k_1}-z^{k_2}-z^{k_1-1}\)
&
$l=0,\ |\mathcal{W}^1|\geq1,\ |\mathcal{B}|\geq2$
\\
\hline

\(P_3(z):=T\left((\mathcal{W}^1_{t+1}\mathcal{W}^2_{k_2})_l;z\right)
=1-z^{k_1}-z^{k_2}-z^m\)
&
$|\mathcal{W}^1|\geq1,\ |\mathcal{W}^2|\geq1$
\\
\hline

\(P_4(z):=T\left((\mathcal{W}^1_{t+1}\mathcal{H})_l;z\right)
=1-z^{k_1}-z^{k_2}-z^{k_1-t+l+1}\)
&
$|\mathcal{W}^1|\geq1,\ (|\mathcal{B}|\geq1 \text{ or } l\neq 0)$
\\
\hline

\(P_5(z):=T\left((\mathcal{W}^1_{t+2}\mathcal{H})_0;z\right)
=1-z^{k_1}-z^{k_2}-z^{k_1-1}\)
&
$l=0,\ |\mathcal{W}^1|\geq2,\ |\mathcal{B}|=0$
\\
\hline

\(P_6(z):=T\left((\mathcal{B}_2\mathcal{W}^1_{k_1})_l;z\right)
=1-z^{k_1}-z^{k_2}-z^{k_1+l-1}\)
&
$|\mathcal{W}^1|\geq1,\ |\mathcal{B}|\geq1$
\\
\hline

\(P_7(z):=T\left((\mathcal{B}_3\mathcal{H})_0;z\right)
=1-z^{k_1}-z^{k_2}-z^{k_1-1}-z^{k_2-1}\)
&
$l=0,\ |\mathcal{B}|\geq2$
\\
\hline

\(P_8(z):=T\left((\mathcal{H}\mathcal{W}^1_{k_1-1})_0;z\right)
=1-z^{k_1}-z^{k_2}-z^{k_1-1}\)
&
$l=0,\ |\mathcal{W}^1|\geq2$
\\
\hline

\(P_9(z):=T\left((\mathcal{H}\mathcal{B}_t)_l;z\right)
=1-z^{k_1}-z^{k_2}-z^{l+t}\)
&
$|\mathcal{B}|\geq1,\ (l\neq0 \text{ or } |\mathcal{W}^2|\geq 1)$
\\
\hline

\(P_{10}(z):=T\left((\mathcal{H}\mathcal{B}_{t-1})_0;z\right)
=
1-z^{k_1}
-z^{k_2}
-z^{t-1}\)
&$l=0,\ |\mathcal{B}|\geq2,\ |\mathcal{W}^2|=0$
\\
\hline

\(P_{11}(z):=T\left((\mathcal{W}^1_i\mathcal{W}^1_{i-2})^+_0;z\right)
=1-z^{k_1}-z^{k_2}-z^{k_1-1}\)
&
$l=0,\ |\mathcal{W}^1|\geq3$
\\
\hline

\(P_{12}(z):=T\left((\mathcal{W}^1_{t+1}\mathcal{W}^1_{k_1})^{-}_l;z\right)
=1-z^{k_1}-z^{k_2}-z^{k_1+l-t}+z^m\)
&
$|\mathcal{W}^1|\geq2$
\\
\hline

\(P_{13}(z):=T\left((\mathcal{B}_i\mathcal{B}_{i-2})^+_0;z\right)
=1-z^{k_1}-z^{k_2}-z^{k_1-1}-z^{k_2-1}\)
&
$l=0,\ |\mathcal{B}|\geq3$
\\
\hline

\(P_{14}(z):=T\left((\mathcal{B}_2\mathcal{B}_t)^{-}_l;z\right)
=1-z^{k_1}-z^{k_2}-z^{l+t-1}\)
&
$|\mathcal{B}|\geq2$
\\
\hline

\(P_{15}(z):=T\left((\mathcal{H}\mathcal{H})_l;z\right)
=1-z^{k_1}-z^{k_2}-z^{l+1}\)
&
$l\neq0$
\\
\hline

\(P_{16}(z):=T\left((\mathcal{H}\mathcal{H})_0;z\right)
=1-z^{k_1}-z^{k_2}-z\)
&
$l=0,\ (|\mathcal{W}^2|\geq1 \text{ or }|\mathcal{B}|\geq1)$
\\
\hline

\(P_{17}(z):=T\left((\mathcal{W}^1_i\mathcal{W}^1_i)_l^0;z\right)
=
1-z^{k_1}-z^{k_2}-z^{l+1}+z^{k_2+l+1}\)
&
$|\mathcal{W}^1|\geq1$
\\
\hline

\(P_{18}(z):=T\left((\mathcal{B}_i\mathcal{B}_i)_l^0;z\right)
=
1-z^{k_1}-z^{k_2}-z^{l+1}\)
&
$|\mathcal{B}|\geq1$

\end{tabular}
\caption{Summary of the $(s,r,i,j)$-\textit{optimized} topological polynomials and their realizability conditions.}
\label{tab:optimized-polynomials}
\end{table}

\begin{lemma}\label{lem:W1>0}
Let $G\in\mathcal{SC}_{m+2}(m)$. Unless a stronger condition is required, we always have
\[
|\mathcal{W}^1|\geq 1.
\]
\end{lemma}

\begin{proof}
Let $G\in\mathcal{SC}_{m+2}(m)$, and suppose, for a contradiction, that
\[
|\mathcal{W}^1|=0.
\]
Then $k_1=t$. Since $k_1\geq k_2$ and $k_2\geq t$, it follows that $k_2=t$.
Thus, before inserting the ear of length $l$, we obtain the digraph
\[
\mathcal{B}_{t,t}^{\,t}.
\]
However, this digraph is not admissible in our setting, since it contains two identical edges $(h_1,b_t)$.
\end{proof}

\begin{lemma}\label{lem:W1>k}
Let $G\in\mathcal{SC}_{m+2}(m)$. Suppose that
\[
|\mathcal{W}^s|\geq k,
\]
where $k\geq 0$ and $s\in\{1,2\}$. Then
\[
k_r\leq m-l-k,
\]
where $r\in\{1,2\}$ and $r\neq s$.
\end{lemma}

\begin{proof}
Since $|\mathcal{W}^s|\geq k$, we have $k_s\geq t+k$. 
Using the relation
\[
k_s+k_r-t+l=m,
\]
we obtain
\[
k_r=m-l-(k_s-t).
\]
Since $k_s-t\geq k$, it follows that
\[
k_r\leq m-l-k.
\]
\end{proof}

\begin{lemma}\label{lem:min}
Let
\[
\mathcal{A}=\{1,2,3,6,7,8,10,11,13\},
\]
and let $P_\alpha(z)$, $\alpha\in\mathcal{A}$, be the polynomials listed in Table~\ref{tab:optimized-polynomials}. Then the polynomial
\[
P_{\min}(z)=1-2z^{m-1}-z^m
\]
corresponds to the topological polynomial of a digraph with minimal spectral radius among the corresponding types. This minimum is attained by the digraphs
\[
(\mathcal{W}^1_{m-1}\mathcal{W}^2_{m-1})_0
\text{-}\mathcal{B}^{m-2}_{m-1,m-1}
\]
and
\[
(\mathcal{B}_2\mathcal{W}^1_m)_0
\text{-}\mathcal{B}^{m-1}_{m,m-1}.
\]
For $m\geq 4$, it is also attained by
\[
(\mathcal{W}^1_m\mathcal{B}_{m-2})_0
\text{-}\mathcal{B}^{m-1}_{m,m-1}.
\]
\end{lemma}

\begin{proof}
For each polynomial $P_\alpha(z)$, $\alpha\in\mathcal{A}$, we maximize the
relevant exponents subject to the structural relation
\[
k_1+k_2-t+l=m
\]
and the corresponding realizability conditions from
Table~\ref{tab:optimized-polynomials}.

We first illustrate the complete optimization procedure for
\[
P_3(z)=1-z^{k_1}-z^{k_2}-z^m,
\]
whose realizability conditions are
$|\mathcal{W}^1|,|\mathcal{W}^2|\geq1$.
The remaining polynomials
$P_\alpha(z)$, $\alpha\in\mathcal{A}\smallsetminus\{3\}$,
are optimized analogously.

For every fixed $z\in(0,1)$, the function $z^a$ is strictly decreasing
with respect to the exponent $a$. Hence, increasing the exponents of the
negative terms increases the corresponding root in $(0,1)$. By
Lemmas~\ref{lem:W1>0} and~\ref{lem:W1>k}, the extremal admissible choice is
\[
k_1=k_2=m-l-1.
\]
For the same reason, the root is further increased by taking the smallest
admissible value of $l$, namely $l=0$. Therefore,
\[
P_3^{\min}(z)=1-2z^{m-1}-z^m.
\]

Applying the same optimization procedure to the remaining polynomials gives
\begin{align*}
    P_1^{\min}(z)&=1-z^{m-2}-z^{m-1}-z^m,
        &\quad (k_1,k_2)&=(m,m-2), \\
    P_2^{\min}(z)&=1-2z^{m-1}-z^m,
        &\quad (k_1,k_2)&=(m,m-1),\\
    P_3^{\min}(z)&=1-2z^{m-l-1}-z^m,
        &\quad (k_1,k_2)&=(m-l-1,m-l-1),\\
    P_6^{\min}(z)&=1-z^{m-l-1}-z^{m-l}-z^{m-1},
        &\quad (k_1,k_2)&=(m-l,m-l-1),\\
    P_7^{\min}(z)&=1-z^{m-2}-2z^{m-1}-z^m,
        &\quad (k_1,k_2)&=(m,m-1),\\
    P_8^{\min}(z)&=1-z^{m-2}-z^{m-1}-z^m,
        &\quad (k_1,k_2)&=(m,m-2),\\
    P_{10}^{\min}(z)&=1-z^{m-2}-z^{m-1}-z^m,
        &\quad (k_1,k_2,t)&=(m,m-1,m-1),\\
    P_{11}^{\min}(z)&=1-z^{m-3}-z^{m-1}-z^m,
        &\quad (k_1,k_2)&=(m,m-3),\\
    P_{13}^{\min}(z)&=1-z^{m-2}-2z^{m-1}-z^m,
        &\quad (k_1,k_2)&=(m,m-1).
\end{align*}

In the cases where the parameter $l$ still occurs, the root is maximized by
taking $l=0$. Hence the above list reduces to
\begin{align*}
    P_1^{\min}(z)&=1-z^{m-2}-z^{m-1}-z^m,
        &\quad (k_1,k_2)&=(m,m-2),\\
    P_2^{\min}(z)&=1-2z^{m-1}-z^m,
        &\quad (k_1,k_2)&=(m,m-1),\\
    P_3^{\min}(z)&=1-2z^{m-1}-z^m,
        &\quad (k_1,k_2)&=(m-1,m-1),\\
    P_6^{\min}(z)&=1-2z^{m-1}-z^m,
        &\quad (k_1,k_2)&=(m,m-1),\\
    P_7^{\min}(z)&=1-z^{m-2}-2z^{m-1}-z^m,
        &\quad (k_1,k_2)&=(m,m-1),\\
    P_8^{\min}(z)&=1-z^{m-2}-z^{m-1}-z^m,
        &\quad (k_1,k_2)&=(m,m-2),\\
    P_{10}^{\min}(z)&=1-z^{m-2}-z^{m-1}-z^m,
        &\quad (k_1,k_2,t)&=(m,m-1,m-1),\\
    P_{11}^{\min}(z)&=1-z^{m-3}-z^{m-1}-z^m,
        &\quad (k_1,k_2)&=(m,m-3),\\
    P_{13}^{\min}(z)&=1-z^{m-2}-2z^{m-1}-z^m,
        &\quad (k_1,k_2)&=(m,m-1).
\end{align*}

It remains to compare the roots of these polynomials. For every
$z\in(0,1)$,
\[
1-z^{m-2}-z^{m-1}-z^m
<
1-2z^{m-1}-z^m,
\]
since $z^{m-2}>z^{m-1}$. Moreover,
\[
1-z^{m-2}-2z^{m-1}-z^m
<
1-2z^{m-1}-z^m,
\]
and
\[
1-z^{m-3}-z^{m-1}-z^m
<
1-2z^{m-1}-z^m.
\]
Since all the corresponding topological polynomials are strictly decreasing
on $(0,1)$, these pointwise inequalities imply that the largest root in
$(0,1)$ is attained by
\[
P_2^{\min}(z)=P_3^{\min}(z)=P_6^{\min}(z)
=
1-2z^{m-1}-z^m.
\]
Thus,
\[
P_{\min}(z):=1-2z^{m-1}-z^m.
\]

The corresponding extremal parameter choices yield the digraphs
\begin{align*}
(\mathcal{W}^1_{m-1}\mathcal{W}^2_{m-1})_0
&\text{-}\mathcal{B}^{\,m-2}_{m-1,m-1},\\
(\mathcal{B}_{2}\mathcal{W}^1_{m})_0
&\text{-}\mathcal{B}^{\,m-1}_{m,m-1}.
\end{align*}
For $m\geq4$, the same polynomial is also attained by
\[
(\mathcal{W}^1_{m}\mathcal{B}_{m-2})_0
\text{-}\mathcal{B}^{\,m-1}_{m,m-1}.
\]

Therefore, $P_{\min}(z)$ has the largest root in $(0,1)$ among the
topological polynomials associated with the types under consideration.
Since the spectral radius is the reciprocal of this root, the corresponding
digraphs have minimal spectral radius among these types.
\end{proof}

\begin{definition}
For $m\geq 3$, we define the \textit{cross-chorded cycle} $\mathcal{C}^\times_m$ as the digraph with vertex set
\[
V(\mathcal{C}^\times_m)=\{v_1,v_2,\dots,v_m\}
\]
and edge set
\[
E(\mathcal{C}^\times_m)
=
\{(v_k,v_{k+1})\mid 1\leq k\leq m-1\}
\cup
\{(v_m,v_1),(v_1,v_3),(v_2,v_4)\},
\]
where the indices are understood cyclically, that is, $v_{m+r}=v_r$.
\end{definition}

\begin{figure}[H]
    \centering
    \includegraphics[width=0.37\textwidth]{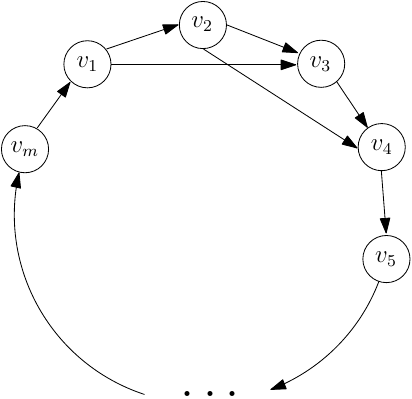}
    \caption{Digraph $\mathcal{C}^\times_m.$}
    \label{fig:min}
\end{figure}

\begin{lemma}\label{lem:isom}
For $m\geq 4$, we have
\[
\mathcal{C}^{\times}_m
\cong
(\mathcal{W}^1_m\mathcal{B}_{m-2})_0
\text{-}\mathcal{B}^{\,m-1}_{m,m-1}
\cong
(\mathcal{W}^1_{m-1}\mathcal{W}^2_{m-1})_0
\text{-}\mathcal{B}^{\,m-2}_{m-1,m-1}
\cong
(\mathcal{B}_2\mathcal{W}^1_m)_0
\text{-}\mathcal{B}^{\,m-1}_{m,m-1}.
\]

For $m=3$, we have
\[
\mathcal{C}^{\times}_3
\cong
(\mathcal{W}^1_2\mathcal{W}^2_2)_0
\text{-}\mathcal{B}^{\,1}_{2,2}
\cong
(\mathcal{B}_2\mathcal{W}^1_3)_0
\text{-}\mathcal{B}^{\,2}_{3,2}.
\]

Consequently, for every $m\geq 3$,
\[
T(\mathcal{C}^{\times}_m;z)
=
P_{\min}(z)
=
1-2z^{m-1}-z^m.
\]
\end{lemma}

\begin{proof}
Let
\begin{align*}
G_1&=(\mathcal{W}^1_{m}\mathcal{B}_{m-2})_0\text{-}\mathcal{B}^{\,m-1}_{m,m-1}, \\
G_2&=(\mathcal{W}^1_{m-1}\mathcal{W}^2_{m-1})_0\text{-}\mathcal{B}^{\,m-2}_{m-1,m-1}, \\
G_3&=(\mathcal{B}_{2}\mathcal{W}^1_{m})_0\text{-}\mathcal{B}^{\,m-1}_{m,m-1}.
\end{align*}
It is enough to give isomorphisms
\[
\phi_1:V(G_2)\to V(G_1),
\quad
\phi_2:V(G_3)\to V(G_1)
\quad \text{and} \quad \phi_3:V(\mathcal{C}^\times_m)\to V(G_1).
\]

Define
\[
\begin{aligned}[t]
\phi_1(v)&=
\begin{cases}
b_{m-1}, & \text{if } v=w^2_{m-1}, \\
w^1_m, & \text{if } v=w^1_{m-1}, \\
v, & \text{otherwise},
\end{cases}
\end{aligned}
\qquad\qquad
\begin{aligned}[t]
\phi_2(v)&=
\begin{cases}
h_1, & \text{if } v=b_2, \\
w^1_m, & \text{if } v=h_1, \\
b_{m-1}, & \text{if } v=w^1_m, \\
b_{\alpha-1}, & \text{if } v=b_\alpha,\ \alpha=3,\dots,m-1,
\end{cases}
\end{aligned}
\]
and
\[
\phi_3(v_\alpha)=
\begin{cases}
h_1, & \text{if } \alpha=1, \\
w^1_m, & \text{if } \alpha=2, \\
b_{m-\alpha+2}, & \text{if } \alpha=3,\dots,m.
\end{cases}
\]

As in the previous isomorphism lemmas, the maps $\phi_1$, $\phi_2$ and $\phi_3$ are digraph isomorphisms. Hence, for $m\geq 4$,
\[
\mathcal{C}^\times_m\cong G_1, \quad
G_2\cong G_1
\quad\text{and}\quad
G_3\cong G_1.
\]
Therefore
\[
\mathcal{C}^\times_m\cong G_1\cong G_2\cong G_3.
\]

For $m=3$, the first digraph $G_1$ is not realizable. The two remaining
digraphs are directly seen to be isomorphic to $\mathcal{C}^{\times}_3$,
and hence
\[
\mathcal{C}^{\times}_3
\cong
(\mathcal{W}^1_2\mathcal{W}^2_2)_0
\text{-}\mathcal{B}^{\,1}_{2,2}
\cong
(\mathcal{B}_2\mathcal{W}^1_3)_0
\text{-}\mathcal{B}^{\,2}_{3,2}.
\]

In both cases, the corresponding topological polynomial is
\[
T(\mathcal{C}^{\times}_m;z)
=
P_{\min}(z)
=
1-2z^{m-1}-z^m.
\]
\end{proof}

\begin{lemma}\label{lem:HH}
The type
\[
(\mathcal{H}\mathcal{H})_0
\]
does not attain the minimum spectral radius in the class
$\mathcal{SC}_{m+2}(m)$.
\end{lemma}

\begin{proof}
For the type $(\mathcal{H}\mathcal{H})_0$, we have $k_1+k_2-t=m$ and the realizability condition
\[
|\mathcal{W}^2|\geq1
\qquad\text{or}\qquad
|\mathcal{B}|\geq1.
\]

First suppose that $|\mathcal{W}^2|\geq1$. Then $k_2-t\geq 1$.
Using $k_1+k_2-t=m$, we obtain
\[
k_1\leq m-1.
\]
Since $k_1\geq k_2$, the exponents are maximized by
\[
k_1=k_2=m-1,
\qquad
t=m-2,
\]
and hence
\[
P_{16}(z)
\leq
1-z-2z^{m-1}.
\]

Now suppose that $|\mathcal{W}^2|=0$. By the realizability condition, we must
have $|\mathcal{B}|\geq1$. Hence $k_2=t\geq2$.
The structural relation $k_1+k_2-t=m$ gives $k_1=m$.
Moreover, by Lemma~\ref{lem:W1>0},
\[
|\mathcal{W}^1|=k_1-t\geq1,
\]
so that $t\leq m-1$.
Therefore the exponents are maximized for
\[
(k_1,k_2,t)=(m,m-1,m-1),
\]
which gives
\[
P_{16}(z)
\leq
1-z-z^{m-1}-z^m.
\]

For every $z\in(0,1)$ and $m\geq3$,
\[
1-z-2z^{m-1}
<
1-z-z^{m-1}-z^m,
\]
because $z^{m-1}>z^m$.
Hence the optimal polynomial for this type is
\[
P_{16}^{\min}(z)
=
1-z-z^{m-1}-z^m.
\]

Finally,
\[
P_{16}^{\min}(z)
<
1-2z^{m-1}-z^m
=
P_{\min}(z),
\]
since $z>z^{m-1}$ for every $z\in(0,1)$ and $m\geq3$.
Thus the root of $P_{16}^{\min}$ in $(0,1)$ is strictly smaller than the root
of $P_{\min}$, and therefore the type $(\mathcal{H}\mathcal{H})_0$ cannot
attain the minimum spectral radius.
\end{proof}

\begin{lemma}\label{lem:omez_moc}
Let $G\in\mathcal{SC}_{m+2}(m)$ be described as one of the types above. Then its parameters $l$ and $t$ satisfy
\[
l+1\leq l+t\leq m-1.
\]
\end{lemma}

\begin{proof}
By Lemma~\ref{lem:W1>0}, we have
\[
|\mathcal{W}^1|\geq 1.
\]
Hence $k_1\geq t+1$. Moreover, since $|\mathcal{W}^2|\geq 0$, we have $k_2\geq t$.
Using the relation
\[
k_1+k_2-t+l=m,
\]
we obtain
\[
m
=
k_1+k_2-t+l
\geq
t+l+1.
\]
Therefore
\[
l+t\leq m-1.
\]

Furthermore, since $|\mathcal{B}|=t-1\geq 0$, we have $t\geq 1$. Combining this with $l+t\leq m-1$, we obtain
\[
l+1\leq l+t\leq m-1.
\]
\end{proof}

\begin{lemma}\label{lem:m=3}
To find the minimum spectral radius in the class $\mathcal{SC}_{m+2}(m)$,
it is necessary to consider the type
\[
(\mathcal{W}^1_{t+1}\mathcal{H})_l
\]
only in the case $m=3$. Moreover, for $m=3$, the optimized digraph
\[
(\mathcal{W}^1_{3}\mathcal{H})_0\text{-}\mathcal{B}^{\,2}_{3,2}
\]
is isomorphic to $\mathcal{C}_3^\times$.
\end{lemma}

\begin{proof}
For this type we have
\[
P_4(z)=1-z^{k_1}-z^{k_2}-z^{k_1-t+l+1}.
\]
The realizability conditions are
\[
|\mathcal{W}^1|\geq1,
\qquad
|\mathcal{B}|\geq1 \ \text{or}\ l\neq0.
\]

First assume that $|\mathcal{B}|\geq1$. Then $t\geq2$. We compare $P_4$ with
\[
P_6(z)=1-z^{k_1}-z^{k_2}-z^{k_1+l-1}.
\]
Since $t\geq2$, we have
\[
k_1-t+l+1\leq k_1+l-1.
\]
Thus, for every $z\in(0,1)$, $P_4(z)\leq P_6(z)$.
By Lemma~\ref{lem:min}, the root of $P_6$ is not larger than the root of
\[
P_{\min}(z)=1-2z^{m-1}-z^m.
\]
Hence $P_4$ can attain the minimum only in the case where equality with the extremal polynomial is possible.

This can occur only if $t=2$. In that case
\[
P_4(z)=1-z^{k_1}-z^{k_2}-z^{k_1+l-1}.
\]
From $k_1+k_2-t+l=m$ we get $k_1+k_2-2+l=m$ and hence
\[
k_1+l-1=m+1-k_2.
\]
Since $k_2\geq t=2$, it follows that $k_1+l-1\leq m-1$.
Also,
\[
k_2\leq m-l-1\leq m-1.
\]
Therefore, if $P_4$ were of the form
\[
P_{\min}(z)=1-2z^{m-1}-z^m,
\]
the exponent $m$ would have to be $k_1$. Hence $k_1=m$. Since
$k_1\leq m-l$, we get $l=0$. Substituting into
\[
k_1+k_2-2+l=m
\]
gives $k_2=2$. Thus the exponents of $P_4$ are
\[
m,\quad 2,\quad m-1.
\]
They coincide with $m,m-1,m-1$ only if $m=3$.

For $m=3$, we obtain
\[
t=2,\qquad l=0,\qquad k_1=3,\qquad k_2=2.
\]
Hence the corresponding digraph is
\[
(\mathcal{W}^1_{3}\mathcal{H})_0\text{-}\mathcal{B}^{\,2}_{3,2},
\]
with topological polynomial
\[
T\left((\mathcal{W}^1_{3}\mathcal{H})_0
\text{-}\mathcal{B}^{\,2}_{3,2};z\right)
=
1-2z^2-z^3.
\]
This is precisely $P_{\min}(z)$ for $m=3$, and the corresponding isomorphism
\[
\phi:
V\left(
\left(\mathcal{W}^1_{3}\mathcal{H}\right)_0
\text{-}\mathcal{B}^{2}_{3,2}
\right)
\to
V\left(\mathcal{C}_3^\times\right)
\]
is given by
\[
\phi(v)=
\begin{cases}
v_1, & \text{if } v=h_1,\\
v_2, & \text{if } v=w_3^1,\\
v_3, & \text{if } v=b_2.
\end{cases}
\]

It remains to consider the case
\[
|\mathcal{B}|=0,\qquad l\neq0.
\]
Then $t=1$ and $l\geq1$. Hence
\[
P_4(z)
=
1-z^{k_1}-z^{k_2}-z^{k_1-t+l+1}
=
1-z^{k_1}-z^{k_2}-z^{k_1+l}.
\]
Since
\[
k_1\leq m-l\leq m-1, \qquad k_2\leq m-l-1\leq m-2,
\]
we also have $k_1+l\leq m$.

Therefore, for every $z\in(0,1)$,
\[
z^{k_1}+z^{k_2}+z^{k_1+l}
\geq
z^{m-1}+z^{m-2}+z^m
>
2z^{m-1}+z^m.
\]
Consequently,
\[
P_4(z)<1-2z^{m-1}-z^m=P_{\min}(z)
\]
for all $z\in(0,1)$. Thus this remaining case cannot attain the minimum.
\end{proof}

\begin{lemma}\label{lem:-1}
To find the minimum spectral radius in the class
$\mathcal{SC}_{m+2}(m)$, it is not necessary to consider the type
\[
(\mathcal{W}^1_{t+2}\mathcal{H})_0.
\]
\end{lemma}

\begin{proof}
For this type, the topological polynomial and the realizability conditions are
$$P_{5}(z) = 1 - z^{k_{1}} - z^{k_{2}} - z^{k_{1}-1}, \quad l=0, \quad |\mathcal{W}^{1}| \ge 2, \quad |\mathcal{B}| = 0.$$

From the structural bounds, we know that $k_{1} \le m$ and $k_{2} \le m-2$. 
By substituting these maximal possible values, we obtain the following bound for $z \in (0,1)$:
$$P_{5}(z) \le 1 - z^{m} - z^{m-2} - z^{m-1}.$$

It follows that
$$1 - z^{m} - z^{m-2} - z^{m-1} < 1 - z^{m} - 2z^{m-1} = P_{\min}(z).$$

Therefore, $P_{5}(z) < P_{\min}(z)$ for all $z \in (0,1)$. Hence the corresponding root is smaller than the root of $P_{\min}(z)$, which implies that this type yields a strictly larger spectral radius and therefore need not be considered when searching for the minimum.
\end{proof}

\begin{lemma}\label{lem:-3}
To find the minimum spectral radius in the class $\mathcal{SC}_{m+2}(m)$,
it is not necessary to consider the types
\[
(\mathcal{H}\mathcal{B}_t)_l,\qquad
(\mathcal{B}_2\mathcal{B}_t)^{-}_l,\qquad
(\mathcal{H}\mathcal{H})_l.
\]
\end{lemma}

\begin{proof}
We prove that, for every admissible choice of parameters,
\[
P_\alpha(z)<P_{\min}(z),\qquad \alpha=9,14,15,
\]
for all $z\in(0,1)$, where
\[
P_{\min}(z)=1-2z^{m-1}-z^m.
\]
Throughout the proof, we use Lemmas~\ref{lem:W1>0}, \ref{lem:W1>k}, and~\ref{lem:omez_moc} to obtain the required bounds.

First, consider
\[
P_9(z)
=
T\left((\mathcal{H}\mathcal{B}_t)_l;z\right)
=
1-z^{k_1}-z^{k_2}-z^{l+t}.
\]
If $l\neq0$, then $l\geq1$, and
\[
k_1\leq m-l\leq m-1,\qquad k_2\leq m-l-1\leq m-2,\qquad l+t\leq m-1.
\]
Hence, for $z\in(0,1)$,
\[
z^{k_1}+z^{k_2}+z^{l+t}
\geq
2z^{m-1}+z^{m-2}
>
2z^{m-1}+z^m.
\]
Therefore
\[
P_9(z)<P_{\min}(z).
\]

If instead $|\mathcal{W}^2|\geq1$, then
\[
k_1\leq m-l-1,\qquad k_2\leq m-l-1,\qquad l+t\leq m-1.
\]
Even for $l=0$, all three exponents are at most $m-1$. Therefore
\[
z^{k_1}+z^{k_2}+z^{l+t}
\geq
3z^{m-1}
>
2z^{m-1}+z^m,
\]
and again
\[
P_9(z)<P_{\min}(z).
\]

Next, consider
\[
P_{14}(z)
=
T\left((\mathcal{B}_2\mathcal{B}_t)^{-}_l;z\right)
=
1-z^{k_1}-z^{k_2}-z^{l+t-1}.
\]
Here
\[
k_1\leq m-l,\qquad k_2\leq m-l-1,\qquad l+t-1\leq m-2.
\]
In particular,
\[
k_1\leq m,\qquad k_2\leq m-1,\qquad l+t-1\leq m-2.
\]
Hence
\[
z^{k_1}+z^{k_2}+z^{l+t-1}
\geq
z^m+z^{m-1}+z^{m-2}
>
z^m+2z^{m-1}.
\]
Thus
\[
P_{14}(z)<P_{\min}(z).
\]

Finally, consider
\[
P_{15}(z)
=
T\left((\mathcal{H}\mathcal{H})_l;z\right)
=
1-z^{k_1}-z^{k_2}-z^{l+1},
\qquad l\neq0.
\]
Since $l\geq1$, the bounds
\[
k_1\leq m-l,\qquad k_2\leq m-l-1,\qquad l+1\leq m-1
\]
give
\[
k_1\leq m-1,\qquad k_2\leq m-2,\qquad l+1\leq m-1.
\]
Therefore
\[
z^{k_1}+z^{k_2}+z^{l+1}
\geq
2z^{m-1}+z^{m-2}
>
2z^{m-1}+z^m,
\]
and hence
\[
P_{15}(z)<P_{\min}(z).
\]

Consequently, for each of the listed types, the corresponding polynomial is strictly smaller than $P_{\min}$ on $(0,1)$. Hence its root in $(0,1)$ is smaller than the root of $P_{\min}$, and these types need not be considered when searching for the minimum.
\end{proof}

\begin{lemma}\label{lem:min-nisom-m=3}
The type
\[
\left(\mathcal{W}^1_{t+1}\mathcal{W}^1_{k_1}\right)_l^{-}
\text{-}\mathcal{B}^{\,t}_{k_1,k_2}
\]
does not attain the minimum spectral radius in the class
$\mathcal{SC}_{m+2}(m)$ for any admissible choice of parameters whenever
$m\geq 4$.

On the other hand, for $m=3$, the digraph
\[
\left(\mathcal{W}^1_{2}\mathcal{W}^1_{3}\right)_0^{-}
\text{-}\mathcal{B}^{\,1}_{3,1}
\]
attains the minimum in the class $\mathcal{SC}_{5}(3)$. Moreover,
\[
\rho\left(
(\mathcal{W}^1_{2}\mathcal{W}^1_{3})^-_0
\text{-}\mathcal{B}^1_{3,1}
\right)
=
\rho(\mathcal{C}_3^\times)
=
\frac{1+\sqrt{5}}{2}.
\]
but
\[
\left(\mathcal{W}^1_{2}\mathcal{W}^1_{3}\right)_0^{-}
\text{-}\mathcal{B}^{\,1}_{3,1}
\ncong
\mathcal{C}^\times_3.
\]
\end{lemma}

\begin{proof}
The topological polynomial of the type
\[
\left(\mathcal{W}^1_{t+1}\mathcal{W}^1_{k_1}\right)_l^{-}
\text{-}\mathcal{B}^{\,t}_{k_1,k_2}
\]
is
\[
P_{12}(z)
=
1-z^{k_1}-z^{k_2}-z^{k_1+l-t}+z^m,
\]
with the realizability condition
\[
|\mathcal{W}^1|\geq 2.
\]

Assume first that $m\geq4$. Since $k_1\leq m-l\leq m$, we have, for every $z\in(0,1)$,
\[
-z^{k_1}+z^m\leq0.
\]
Hence
\[
P_{12}(z)
\leq
1-z^{k_2}-z^{k_1+l-t}.
\]
Using the relation
\[
k_1+k_2-t+l=m,
\]
we obtain
\[
k_1+l-t=m-k_2,
\]
and therefore
\[
P_{12}(z)
\leq
1-z^{k_2}-z^{m-k_2}.
\]
By the arithmetic-geometric mean inequality,
\[
z^{k_2}+z^{m-k_2}\geq 2z^{m/2}.
\]
Therefore,
\[
P_{12}(z)\leq 1-2z^{m/2}.
\]

Let $z_0\in(0,1)$ be the root of $P_{\min}$. That is
\begin{align*}
1-2z_0^{m-1}-z_0^m&=0, \\
z_0^{m-1}(2+z_0)&=1,
\end{align*}
and hence
\begin{equation}\label{z_0-to-m}
z_0^{m-1}=\frac{1}{2+z_0}.
\end{equation}
Multiplying by $z_0$, we get
\begin{equation}\label{z_0}
z_0^m=\frac{z_0}{2+z_0}.
\end{equation}

We now show that $z_0>2/3$ for every $m\geq4$. Indeed,
\[
P_{\min}\left(\frac23\right)
=
1-2\left(\frac23\right)^{m-1}
-\left(\frac23\right)^m
=
1-4\left(\frac23\right)^m
>0
\]
for all $m\geq4$. Therefore $z_0>2/3$. Substituting this estimate into
\eqref{z_0}, we obtain
\[
z_0^m>\frac14.
\]
Consequently,
\[
2z_0^{m/2}
>
2\sqrt{\frac14}
=
1.
\]
Thus
\begin{equation}\label{AM-GM}
P_{12}(z_0)
\leq
1-2z_0^{m/2}
<
0.
\end{equation}
Since $P_{12}(0)=1$, the root of $P_{12}$ in $(0,1)$ is smaller than
$z_0$, the root of $P_{\min}$. Hence this type cannot attain the minimum
for $m\geq4$.

It remains to consider the case $m=3$. The condition
\[
|\mathcal{W}^1|\geq2
\]
implies $k_2\leq m-l-2=1-l$. Since $t\geq1$ and $k_2\geq t$, we must have
\[
l=0,\qquad k_2=1.
\]
Then necessarily
\[
t=1,\qquad k_1=3.
\]
Thus, for $m=3$, the only admissible choice of parameters is
\[
(k_1,k_2,t,l)=(3,1,1,0).
\]
This choice corresponds to the digraph
\[
\left(\mathcal{W}^1_{2}\mathcal{W}^1_{3}\right)_0^{-}
\text{-}\mathcal{B}^{\,1}_{3,1},
\]
whose topological polynomial is
\[
T\left(
\left(\mathcal{W}^1_{2}\mathcal{W}^1_{3}\right)_0^{-}
\text{-}\mathcal{B}^{\,1}_{3,1};z
\right)
=
1-z-z^2.
\]

For $m=3$, we have
\[
P_{\min}(z)
=
T(\mathcal{C}^\times_3;z)
=
1-2z^2-z^3
=
(1+z)(1-z-z^2).
\]
Therefore the two polynomials
\[
T\left(
\left(\mathcal{W}^1_{2}\mathcal{W}^1_{3}\right)_0^{-}
\text{-}\mathcal{B}^{\,1}_{3,1};z
\right) \quad \text{ and } \quad T(\mathcal{C}^\times_3;z) 
\]
have the same unique root  $R\in(0,1)$, namely
\[
R=\frac{\sqrt5-1}{2}.
\]
Hence the corresponding digraphs have the same spectral radius
\[
\rho=R^{-1}=\frac{1+\sqrt5}{2}.
\]

Finally, the two digraphs are not isomorphic. Indeed, the digraph
\[
\left(\mathcal{W}^1_{2}\mathcal{W}^1_{3}\right)_0^{-}
\text{-}\mathcal{B}^{\,1}_{3,1}
\]
contains the loop $(h_1,h_1)$ (Figure~\ref{fig:min2}), whereas $\mathcal{C}^\times_3$ contains no loop.
Thus
\[
\left(\mathcal{W}^1_{2}\mathcal{W}^1_{3}\right)_0^{-}
\text{-}\mathcal{B}^{\,1}_{3,1}
\ncong
\mathcal{C}^\times_3.
\]
\end{proof}

\begin{figure}[H]
    \centering
    \includegraphics[width=0.18\textwidth]{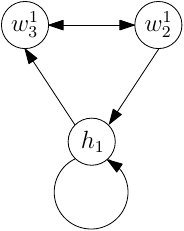}
    \caption{Digraph $\left(\mathcal{W}^1_{2}\mathcal{W}^1_{3}\right)_0^{-}
\text{-}\mathcal{B}^{\,1}_{3,1}.$}
    \label{fig:min2}
\end{figure}

\begin{lemma}\label{lem:-1_min}
To find the minimum spectral radius in the class $\mathcal{SC}_{m+2}(m)$,
the cycle type
\[
(\mathcal{B}_i\mathcal{B}_i)^0_l
\]
need not be considered.
The cycle type
\[
(\mathcal{W}^1_i\mathcal{W}^1_i)^0_l
\]
also does not attain the minimum for $m\geq 4$. For $m=3$, the only exceptional case is
\[
(\mathcal{W}^1_3\mathcal{W}^1_3)_0^0
\text{-}\mathcal{B}^{\,2}_{3,2}.
\]
This digraph has the same spectral radius as the minimizers already obtained for
$m=3$, and it is isomorphic to
\[
(\mathcal{W}^1_{2}\mathcal{W}^1_{3})^-_0
\text{-}\mathcal{B}^1_{3,1}.
\]
\end{lemma}

\begin{proof}
For the given types, the corresponding topological polynomials and realizability
conditions are
\[
P_{17}(z)
=
T\left((\mathcal{W}^1_i\mathcal{W}^1_i)_l^0;z\right)
=
1-z^{k_1}-z^{k_2}-z^{l+1}+z^{k_2+l+1},
\qquad
|\mathcal{W}^1|\geq1,
\]
and
\[
P_{18}(z)
=
T\left((\mathcal{B}_i\mathcal{B}_i)_l^0;z\right)
=
1-z^{k_1}-z^{k_2}-z^{l+1},
\qquad
|\mathcal{B}|\geq1.
\]

There is no condition on $\mathcal{B}$ for the polynomial $P_{17}$, so
$|\mathcal{B}|\geq0$ is sufficient. On the other hand, by
Lemma~\ref{lem:W1>0}, every admissible digraph corresponding to $P_{18}$
also satisfies $|\mathcal{W}^1|\geq1$. Hence, whenever $P_{18}$ is
realizable, $P_{17}$ is realizable with the same parameters. Moreover, for
fixed parameters and every $z\in(0,1)$, we have
\[
P_{18}(z)
=
1-z^{k_1}-z^{k_2}-z^{l+1}
<
1-z^{k_1}-z^{k_2}-z^{l+1}+z^{k_2+l+1}
=
P_{17}(z).
\]
Therefore the smallest positive root of $P_{18}$ is smaller than that of
$P_{17}$. Thus the type
\[
(\mathcal{B}_i\mathcal{B}_i)_l^0
\]
cannot improve the minimum and need not be considered.

It remains to analyze
\[
P_{17}(z)
=
1-z^{k_1}-z^{k_2}-z^{l+1}+z^{k_2+l+1}
=
1-z^{k_1}-z^{k_2}-z^{l+1}(1-z^{k_2}),
\qquad
|\mathcal{W}^1|\geq1.
\]
We compare it with
\[
P_{12}(z)
=
1-z^{k_1}-z^{k_2}-z^{k_1+l-t}+z^m
=
1-z^{k_1}-z^{k_2}-z^{k_1+l-t}(1-z^{k_2}),
\qquad
|\mathcal{W}^1|\geq2.
\]
Assume first that $|\mathcal{W}^1|\geq2$. Then $k_1-t\geq2$, and hence
\[
l+1<l+2\leq k_1-t+l.
\]
Since $z\in(0,1)$, it follows that
\[
P_{17}(z)
=
1-z^{k_1}-z^{k_2}-z^{l+1}(1-z^{k_2})
<
1-z^{k_1}-z^{k_2}-z^{k_1+l-t}(1-z^{k_2})
=
P_{12}(z).
\]
Thus, in this case, the smallest positive root of $P_{17}$ is smaller than
that of $P_{12}$, so $P_{17}$ cannot attain the minimum.

It remains to consider the case $|\mathcal{W}^1|=1$. Then $k_1=t+1$.
Since $t\leq k_2\leq k_1$, we have either $k_2=t$ or $k_2=t+1$. Write
\[
k_2=t+\alpha,
\qquad
\alpha\in\{0,1\}.
\]
Using the relation $k_1+k_2-t+l=m$, we obtain
\begin{equation}\label{alpha_rel}
t+\alpha+l+1=m.
\end{equation}
Consequently,
\[
P_{17}(z)
=
1-z^{t+1}-z^{t+\alpha}-z^{l+1}+z^m.
\]

Let $z_0\in(0,1)$ be the root of $P_{\min}(z)$, that is,
\[
1-2z_0^{m-1}-z_0^m=0.
\]
By \eqref{alpha_rel} and the AM-GM inequality,
\[
z_0^{t+\alpha}+z_0^{l+1}\geq 2z_0^{m/2}.
\]
Moreover, \eqref{alpha_rel} gives
\[
m-(t+1)=\alpha+l\geq0,
\]
and hence $t+1\leq m$. Therefore $z_0^{t+1}\geq z_0^m$. Combining these
estimates gives
\[
P_{17}(z_0)
=
1-z_0^{t+1}-z_0^{t+\alpha}-z_0^{l+1}+z_0^m
\leq
1-2z_0^{m/2}.
\]
By \eqref{AM-GM}, for $m\geq4$ we have
\[
1-2z_0^{m/2}<0.
\]
Thus $P_{17}(z_0)<0$, and so the smallest positive root of $P_{17}$ is
smaller than $z_0$. Hence $P_{17}$ cannot attain the minimum for
$m\geq4$.

It remains to treat the case $m=3$. Then
\[
P_{\min}(z_0)=1-2z_0^2-z_0^3=(1+z_0)(1-z_0-z_0^2)=0,
\]
and since $z_0\in(0,1)$, we get
\begin{equation}\label{odhad}
1-z_0-z_0^2=0.
\end{equation}
Because $t\geq1$, equation \eqref{alpha_rel} gives only the following three
possibilities for $(\alpha,t,l)$:
\[
(0,1,1),\qquad (1,1,0),\qquad \text{and}\qquad (0,2,0).
\]

In the cases $(1,1,0)$ and $(0,1,1)$, we obtain
\[
P_{17}(z_0)=1-z_0-2z_0^2+z_0^3.
\]
Using \eqref{odhad}, this becomes
\[
P_{17}(z_0)=z_0^2(z_0-1)<0.
\]
Hence these two cases do not attain the minimum.

In the remaining case $(0,2,0)$, we obtain
\[
P_{17}(z_0)=1-z_0-z_0^2=0.
\]
Thus the digraph
\[
(\mathcal{W}^1_3\mathcal{W}^1_3)_0^0
\text{-}\mathcal{B}^{\,2}_{3,2}
\]
has the same spectral radius as the already obtained minimizers
\[
\mathcal{C}^\times_3
\quad\text{and}\quad
\left(\mathcal{W}^1_{2}\mathcal{W}^1_{3}\right)^{0^-}_0
\text{-}\mathcal{B}^{\,1}_{3,1}.
\]

Finally, this exceptional digraph is isomorphic to
\[
\left(\mathcal{W}^1_{2}\mathcal{W}^1_{3}\right)^{-}_0
\text{-}\mathcal{B}^{\,1}_{3,1}.
\]
Indeed, define
\[
\phi:
V\left(
\left(\mathcal{W}^1_{2}\mathcal{W}^1_{3}\right)^{-}_0
\text{-}\mathcal{B}^{\,1}_{3,1}
\right)
\to
V\left(
(\mathcal{W}^1_3\mathcal{W}^1_3)_0^0
\text{-}\mathcal{B}^{\,2}_{3,2}
\right)
\]
by
\[
\phi(v)=
\begin{cases}
w^1_3, & \text{if } v=h_1, \\
b_2, & \text{if } v=w^1_3, \\
h_1, & \text{if } v=w^1_2.
\end{cases}
\]
\end{proof}

\begin{lemma}\label{lem:bound}
Let $m\geq 3$, and let $\mathcal{C}_m^\times$ be defined as above. Then
\[
2^{\frac{1}{m-1}}
<
\rho\left(\mathcal{C}_m^\times\right)
<
3^{\frac{1}{m-1}}.
\]
\end{lemma}

\begin{proof}
Let $z_0\in(0,1)$ be the root of
\[
T(\mathcal{C}_m^\times;z)=1-2z^{m-1}-z^m.
\]
Then, by \eqref{z_0-to-m},
\[
z_0^{m-1}=\frac{1}{2+z_0}.
\]
Since $z_0\in(0,1)$ and $\rho\left(\mathcal{C}_m^\times\right)=z_0^{-1}$, we have
\[
\frac13<z_0^{m-1}<\frac12,
\]
and therefore
\[
2^{\frac{1}{m-1}}
<
\rho\left(\mathcal{C}_m^\times\right)
<
3^{\frac{1}{m-1}}.
\]
\end{proof}

\begin{theorem}
Let $m\geq 3$, and let $G\in\mathcal{SC}_{m+2}(m)$. Let $R_m\in(0,1)$ be the unique
root in $(0,1)$ of
\[
P_{\min}(z)=1-2z^{m-1}-z^m.
\]
Then
\[
\rho(G)\geq R_m^{-1},
\]
and hence
\[
\min_{G\in\mathcal{SC}_{m+2}(m)}\rho(G)
=
R_m^{-1}
=
\rho\left(\mathcal{C}_m^\times\right).
\]
Moreover,
\[
2^{\frac{1}{m-1}}
<
\rho\left(\mathcal{C}_m^\times\right)
<
3^{\frac{1}{m-1}}.
\]

If $m\geq 4$, equality
\[
\rho(G)=R_m^{-1}
\]
holds if and only if
\[
G\cong \mathcal{C}_m^\times.
\]
In the notation introduced above, this extremal isomorphism class is represented by
\[
(\mathcal{W}^1_m\mathcal{B}_{m-2})_0
\text{-}\mathcal{B}^{m-1}_{m,m-1},
\]
\[
(\mathcal{W}^1_{m-1}\mathcal{W}^2_{m-1})_0
\text{-}\mathcal{B}^{m-2}_{m-1,m-1},
\]
and
\[
(\mathcal{B}_2\mathcal{W}^1_m)_0
\text{-}\mathcal{B}^{m-1}_{m,m-1},
\]
all of which are isomorphic to $\mathcal{C}_m^\times$.

If $m=3$, equality holds if and only if $G$ is isomorphic to one of the two non-isomorphic digraphs
\[
\mathcal{C}_3^\times
\]
and
\[
(\mathcal{W}^1_2\mathcal{W}^1_3)^-_0
\text{-}\mathcal{B}^1_{3,1}.
\]
Both have spectral radius
\[
\rho(\mathcal{C}_3^\times)
=
\rho\left(
(\mathcal{W}^1_2\mathcal{W}^1_3)^-_0
\text{-}\mathcal{B}^1_{3,1}
\right)
=
\frac{1+\sqrt{5}}{2},
\]
but
\[
(\mathcal{W}^1_2\mathcal{W}^1_3)^-_0
\text{-}\mathcal{B}^1_{3,1}
\not\cong
\mathcal{C}_3^\times.
\]
\end{theorem}

\begin{proof}
In Sections~3 and~4, we showed that every digraph
$G\in\mathcal{SC}_{m+2}(m)$ can be represented in the form
\[
(\mathcal{X}_i\mathcal{Y}_j)_l
\text{-}\mathcal{B}^{\,t}_{k_1,k_2}, \quad \text{or} \quad (\mathcal{X}_i\mathcal{X}_j)_l^{0/+/-}
\text{-}\mathcal{B}^{\,t}_{k_1,k_2},
\]
where
\[
\mathcal{X},\mathcal{Y}\in
\{\mathcal{W}^1,\mathcal{W}^2,\mathcal{B},\mathcal{H}\}.
\]
Moreover, for each such type we derived the corresponding
$(s,r,i,j)$-optimized topological polynomial for a candidate minimizer
within that type. All eighteen optimized polynomials are listed in
Table~\ref{tab:optimized-polynomials}.

In Lemma~\ref{lem:min}, we considered nine of these types and showed that the
extremal topological polynomial is
\[
P_{\min}(z)=1-2z^{m-1}-z^m.
\]
For $m\geq4$, this polynomial is attained by the digraphs
\[
(\mathcal{W}^1_m\mathcal{B}_{m-2})_0
\text{-}\mathcal{B}^{\,m-1}_{m,m-1},
\]
\[
(\mathcal{W}^1_{m-1}\mathcal{W}^2_{m-1})_0
\text{-}\mathcal{B}^{\,m-2}_{m-1,m-1},
\]
and
\[
(\mathcal{B}_2\mathcal{W}^1_m)_0
\text{-}\mathcal{B}^{\,m-1}_{m,m-1}.
\]
By Lemma~\ref{lem:isom}, all three of these digraphs are isomorphic to
$\mathcal{C}^\times_m$.

For $m=3$, the same polynomial is attained by
\[
(\mathcal{W}^1_2\mathcal{W}^2_2)_0
\text{-}\mathcal{B}^{\,1}_{2,2}
\]
and
\[
(\mathcal{B}_2\mathcal{W}^1_3)_0
\text{-}\mathcal{B}^{\,2}_{3,2},
\]
and Lemma~\ref{lem:isom} shows that both are isomorphic to
$\mathcal{C}^\times_3$.

Lemma~\ref{lem:HH} shows that one further type does not attain the minimum spectral radius.
Lemma~\ref{lem:m=3} shows that another candidate attains the same minimum
only in the exceptional case $m=3$, and that in this case the corresponding
digraph is again isomorphic to $\mathcal{C}^\times_3$.
Lemmas~\ref{lem:-1} and~\ref{lem:-3} show that four further candidates
cannot attain the minimum.

Lemma \ref{lem:min-nisom-m=3} shows that, for $m=3$, the digraph
\[
(\mathcal{W}^1_2\mathcal{W}^1_3)^-_0
\text{-}\mathcal{B}^1_{3,1}
\]
has the same spectral radius as $\mathcal{C}_3^\times$, but is not isomorphic to it. Moreover,
\[
\rho(\mathcal{C}_3^\times)
=
\rho\left(
(\mathcal{W}^1_2\mathcal{W}^1_3)^-_0
\text{-}\mathcal{B}^1_{3,1}
\right)
=
\frac{1+\sqrt{5}}{2}.
\]
On the other hand, for $m\geq4$, the type
\[
\left(\mathcal{W}^1_{t+1}\mathcal{W}^1_{k_1}\right)_l^-
\text{-}\mathcal{B}^{\,t}_{k_1,k_2}
\]
never attains the same minimal spectral radius as
$\mathcal{C}^\times_m$.

Finally, Lemma~\ref{lem:-1_min} shows that one further type does not attain
the minimum, while the other attains the minimum only in the case $m=3$.
In this exceptional case, the corresponding digraph is
\[
(\mathcal{W}^1_3\mathcal{W}^1_3)_0^0
\text{-}\mathcal{B}^{\,2}_{3,2},
\]
which is isomorphic to
\[
\left(\mathcal{W}^1_2\mathcal{W}^1_3\right)_0^-
\text{-}\mathcal{B}^{\,1}_{3,1}.
\]

Thus all eighteen optimized candidates for the minimum listed in Table~\ref{tab:optimized-polynomials} have
been considered. It follows that, for $m\geq 4$, the minimum spectral radius
in the class $\mathcal{SC}_{m+2}(m)$ is attained precisely by the isomorphism
class of $\mathcal{C}_m^\times$. For $m=3$, there are two non-isomorphic
minimizers, namely
\[
\mathcal{C}_3^\times
\]
and
\[
(\mathcal{W}^1_2\mathcal{W}^1_3)^-_0
\text{-}\mathcal{B}^1_{3,1}.
\]
The bounds
\[
2^{\frac{1}{m-1}}
<
\rho\left(\mathcal{C}_m^\times\right)
<
3^{\frac{1}{m-1}}
\]
were proved in Lemma~\ref{lem:bound}.
\end{proof}

\normalfont
The structural classification obtained in this paper also suggests a natural
extremal problem at the opposite end of the spectral-radius spectrum.
Recall that minimizing the spectral radius is equivalent to maximizing the
smallest positive root of the corresponding topological polynomial.
Conversely, maximizing the spectral radius amounts to minimizing this root.

For the polynomial families obtained in Section~4, this suggests making the
exponents of the negative monomials as small as possible, subject to the
structural relation and the corresponding realizability conditions. This leads to the
following conjecture.

\begin{conjecture}
Let $m\geq3$.

\begin{enumerate}
    \item Among all digraphs in $\mathcal{SC}_{m+2}(m)$, the maximum spectral
    radius is attained, up to isomorphism, by the digraph represented by
    \[
        (\mathcal{H}\mathcal{H})_l
        \text{-}\mathcal{B}_{k_1,k_2}^{\,t}
    \]
    with
    \[
        (k_1,k_2,t,l)
        \in
        \bigl\{
        (2,1,1,m-2),
        (m-1,1,1,1),
        (m-1,2,1,0)
        \bigr\}.
    \]
    Its topological polynomial is
    \[
        P_{\max}(z)=1-z-z^2-z^{m-1}.
    \]
    Hence, if $R_m^{\max}\in(0,1)$ denotes its smallest positive root, then
    \[
        \max_{G\in\mathcal{SC}_{m+2}(m)}\rho(G)
        =
        \left(R_m^{\max}\right)^{-1}.
    \]

    \item For every $m\geq4$, among all loopless digraphs in
    $\mathcal{SC}_{m+2}(m)$, the maximum spectral radius is attained, up to
    isomorphism, by the digraph represented by
    \[
        (\mathcal{H}\mathcal{H})_l
        \text{-}\mathcal{B}_{k_1,k_2}^{\,t}
    \]
    with
    \[
        (k_1,k_2,t,l)
        \in
        \bigl\{
        (2,2,1,m-3),
        (m-2,2,1,1)
        \bigr\}.
    \]
    Its topological polynomial is
    \[
        \widehat P_{\max}(z)=1-2z^2-z^{m-2}.
    \]
    Hence, if $\widehat R_m^{\max}\in(0,1)$ denotes its smallest positive root,
    then
    \[
        \max_{\substack{G\in\mathcal{SC}_{m+2}(m)\\ G\text{ loopless}}}
        \rho(G)
        =
        \left(\widehat R_m^{\max}\right)^{-1}.
    \]
\end{enumerate}
\end{conjecture}

For $m=3$, the loopless case is exceptional. Up to isomorphism,
$\mathcal{SC}_{5}(3)$ contains only one loopless digraph. Consequently, its
maximum and minimum spectral radii coincide and are equal to
\[
    \frac{1+\sqrt{5}}{2}.
\]

\section{Acknowledgements}
Research was funded by institutional support for the development of research organisations (I\v{C} 47813059) and by Grant SGS 16/2024.

\pagebreak
\bibliographystyle{plain}
\bibliography{refs}

\end{document}